\documentclass[11pt]{amsart}
\usepackage[utf8]{inputenc}
\usepackage[margin=1in]{geometry}
\usepackage{float,graphicx,subcaption}
\usepackage{amsfonts,amsmath,amssymb,amsthm,bbm,mathtools}
\usepackage{enumitem,nameref}
\usepackage{hyperref}
\usepackage{xcolor}

\newtheorem{theorem}{Theorem}[section]
\newtheorem{lemma}[theorem]{Lemma}
\newtheorem{proposition}[theorem]{Proposition}

\theoremstyle{definition}

\newtheorem{remark}[theorem]{Remark}
\newtheorem{assumption}{Assumption}

\title[Synchronization in random networks]{Synchronization in random networks of identical phase oscillators: A graphon approach}

\author{Shriya V.~Nagpal}
\author{Gokul G.~Nair}
\author{Steven H.~Strogatz}
\author{Francesca Parise}

\address{Mathematics, Pitzer College, 1050 N Mills Ave, Claremont, CA 91711}
\email{shriyn@pitzer.edu}

\address{Department of Mathematics, Rutgers University, 110 Frelinghuysen Road, Piscataway, NJ 08854}
\email{gokul.nair@rutgers.edu}

\address{Center for Applied Mathematics, Cornell University, Ithaca, NY 14853}
\email{svn23@cornell.edu}
\email{gn234@cornell.edu}
\email{shs7@cornell.edu}

\address{School of Electrical and Computer Engineering, Cornell University, Ithaca, NY 14853}
\email{fp264@cornell.edu}

\keywords{Graphons, Random graph, Synchronization, Kuramoto model, Continuum limit}

\subjclass{34C15, 34D06, 45K05, 05C80}

\begin{document}

\begin{abstract}
    Networks of coupled nonlinear oscillators have been used to model circadian rhythms, flashing fireflies, Josephson junction arrays, high-voltage electric grids,  and many other kinds of self-organizing systems. Recently, several authors have sought to understand how coupled oscillators behave when they interact according to a random graph. Here we consider interaction networks generated by a graphon model known as a $W$-random network, and examine the dynamics of an infinite number of identical phase oscillators. We show that with sufficient regularity on $W$, the solution to the dynamical system over a $W$-random network of size $n$ converges in the $L^{\infty}$ norm to the solution of the infinite graphon system, with high probability as $n\rightarrow\infty$. We leverage this convergence result to explore synchronization for two classes of identical phase oscillators on Erd\H{o}s-R\'enyi random graphs. This result suggests a framework for studying synchronization properties in large but finite random networks. 
\end{abstract}

\maketitle

\section{Introduction}
Networks of phase oscillators have received a great deal of attention recently, in part because of their many applications in physics, biology, chemistry, and engineering, and also because of the fascinating mathematical issues they raise about spontaneous synchronization, chimera states, and other forms of collective behavior~\cite{Acebron2005, parastesh2021chimeras, Pikovsky2001, Rodrigues2016, Strogatz2003, strogatz2018nonlinear}. 

Two of the best-studied examples of phase oscillator models are the Kuramoto model~\cite{Acebron2005,Kuramoto1975,Kuramoto1984} and the Sakaguchi-Kuramoto model~\cite{sakaguchi1986solubl}. To describe each, let $A^{(n)}\in\mathbb{R}^{n\times{n}}$ be the adjacency matrix associated with an unweighted and undirected network on $n$ nodes, $G=(V,E)$, where $A^{(n)}_{ij}=A^{(n)}_{ji}=1$ if and only if $(i,j)\in{E}$, and $A^{(n)}_{ij}=A^{(n)}_{ji}=0$ otherwise. In the \textit{Kuramoto model}, the state of each node $i\in{V}$ is given by a phase angle $\theta_i(t)$ that evolves according to the following system of ordinary differential equations:
$$
\dot{\theta_{i}}(t)=\nu_i + \sum_{j=1}^n A^{(n)}_{ij} \sin \left(\theta_j(t)-\theta_i(t)\right)
$$

\noindent for $i=1,\ldots, n$. Here, the overdot denotes differentiation with respect to time $t$, and $\nu_i$ is the natural frequency of oscillator $i$. 
The \textit{Sakaguchi-Kuramoto model} extends the Kuramoto model by introducing a \textit{phase shift parameter}, $0<\beta<\frac{\pi}{2}$. The governing equations become: 
\[
\dot{\theta_{i}}(t)=\nu_i + \sum_{j=1}^n A^{(n)}_{ij} \sin \left(\theta_j(t)-\theta_i(t) + \beta\right)
\]
\noindent for $i=1,\ldots, n$. 
For both models, a main question is whether the oscillators will synchronize or settle into some other form of long-term behavior. 

Because the Kuramoto model was originally inspired by statistical physics, most of the early work on it assumed that the interaction network was a structured lattice, such as a one-dimensional chain or ring, a two-dimensional square grid, or a cubic lattice of dimension three or higher~\cite{sakaguchi1987local, sakaguchi1988mutual, strogatz1988phase}. In those settings, the natural frequencies $\nu_i$ were usually assumed to be randomly distributed across the nodes according to some prescribed probability distribution, and the main question was whether the system would undergo a phase transition to a macroscopically synchronized state as the variance of the frequencies was reduced. 

Recent research has complemented these studies by exploring the behavior of these models on random graphs \cite{Abdalla, Kassabov, Ling, Mehtaab}. For simplicity, suppose that the oscillators have identical frequencies (a case known as the ``homogeneous'' model~\cite{Taylor2012}). Then, by going into a rotating frame, one can set $\nu_i = 0$ for all $i$ without loss of generality. Indeed, we will assume $\nu_i = 0$ for all $i$ from now on. A key question in this setting is how the topology of the network affects its tendency to synchronize. 

Two types of synchronization, \textit{phase synchronization} and \textit{frequency synchronization}, are of particular interest in this context. Phase synchronization is the strongest possible notion of synchrony; it means that the oscillators asymptotically approach the same phase. Frequency synchronization means that the oscillators asymptotically move at the same constant frequency.

\subsection{Graphons}

The recent literature has used graphon theory~\cite{Medvedev_2014,Medvedev_2013,Medvedev_main} to study random networks of oscillators and their synchronization transition in the large-$n$ limit. Mathematically, a \textit{graphon} is a symmetric measurable function on the unit square.  Intuitively, a graphon can be interpreted as the continuum limit of the adjacency matrix of an undirected graph as its size tends to infinity \cite{Borgs,Lovasz}. Building on this interpretation, we can define a continuum network of oscillators where each element in the interval $[0,1]$ labels an oscillator whose behavior is governed by an integro-partial differential equation with interactions dictated by the graphon. 

A second interpretation of a graphon is as a random graph model \cite{Lovasz}. Here one constructs a ``sampled'' adjacency matrix (also known as a $W$-random network) of size $n$ from the graphon. The probability that two nodes are adjacent to one another is determined by a discretization of the graphon weighted by a \textit{scaling factor}, $\alpha_n$, where $0<\alpha_n\leq{1}$ for all $n$. In this second interpretation, sampled network dynamics are defined such that the oscillator at each node is governed by an ordinary differential equation, with interactions dictated by the sampled adjacency matrix. To understand the behavior of large random networks of oscillators, our strategy is to prove a convergence result that relates the solution of the \textit{continuum dynamics} to the \textit{sampled dynamics} as $n$ goes to infinity. 


\subsection{Contributions}

Our analysis yields two main contributions to the study of oscillators on random graphs. The first is a convergence result. We prove that when the graphon $W$ is continuous, a piecewise interpolant of the solution to the sampled dynamical system of size $n$ converges in the $L^{\infty}$ norm to the solution of the continuous graphon dynamical system,  with high probability as $n\rightarrow\infty$. Second, we apply this convergence result to study synchronization properties of identical phase oscillators. We focus on two oscillator models, the Kuramoto and Sakaguchi-Kuramoto oscillators, interacting over Erd\H{o}s-R\'enyi random networks.

\subsection{Relation to Previous Work}
\sloppy

Our convergence result is related to recent studies initiated by Medvedev~\cite{Medvedev_2014, Medvedev_2013, Medvedev_main} and further explored by Bramburger, Holzer, and Williams in \cite{bramburger2024}. Assuming a bounded, symmetric, and almost everywhere continuous graphon, Medvedev proves convergence of the sampled dynamics to the continuous dynamics (as defined above) as $n$ goes to infinity in the $L^{2}$ norm with high probability \cite{Medvedev_2014}. In \cite{Medvedev_main}, Medvedev loosens the regularity assumptions on the graphon and strengthens the convergence results in the $L^{2}$ norm by proving convergence with probability $1$ as $n$ goes to infinity. Another key distinction between \cite{Medvedev_2014} and \cite{Medvedev_main} is that the framework  in \cite{Medvedev_main} allows for the sparse graphon regime with $\alpha_n=\omega(n^{-1/2})$. 

Our work also operates under the sparse graphon regime ($\frac{\log (n) / n}{\alpha_n^3} \rightarrow 0$ as $n\rightarrow\infty$) and shows convergence with high probability in the $L^{\infty}$ norm, albeit via stronger regularity assumptions on the graphon. Proving convergence in the $L^{\infty}$ norm, instead of $L^{2}$, is necessary in our work to study the synchronization properties of Kuramoto and Sakaguchi-Kuramoto oscillators interacting over Erd\H{o}s--R\'enyi random networks.

\fussy
Our analysis is also related to recent studies of \textit{global synchronization} \cite{Abdalla,Kassabov,Ling}. A network of oscillators is said to globally synchronize if it converges to a state with all the oscillators in phase, starting from any initial condition except for a set of measure zero \cite{Kassabov}. In \cite{Bandeira}, it was conjectured that a system of $n$ identical Kuramoto oscillators on an Erd\H{o}s-R\'enyi random graph would globally synchronize with high probability for $p$ right above the connectivity threshold, that is, for $p =(1+\epsilon)\left(\frac{\log(n)}{n}\right)$ where $\epsilon>0$. Ling et al. \cite{Ling} took the first step in this direction by proving that global synchronization occurs with high probability for $p=\frac{\Omega\left(\log(n)\right)}{n^{1/3}}$. This result was later improved to $p=\frac{{\Omega\left(\log(n)\right)}^{2}}{n}$  by Kassabov et al.~\cite{Kassabov}, and finally the  original conjecture was proven in~\cite{Abdalla}. 

Our work offers a different perspective. Instead of focusing directly on the finite-$n$ case, we drive a connection with the continuum  model. By setting our graphon equal to $1$, the resulting sparse $W$-random network coincides with an Erd\H{o}s-R\'enyi model where the probability  of an edge existing between two nodes is $\alpha_n$ (where $\alpha_n\rightarrow{0}$ as $n\rightarrow{\infty}$). In this case, synchronization properties of the continuum Kuramoto model are inherited with high probability by the sampled networks. We show that when the continuum model synchronizes, the oscillators on the finite-$n$ Erd\H{o}s-R\'enyi random graph attain phases that get close to each other---specifically, within a distance of $\pi$ of each other---at a fixed time, with high probability for large $n$. We then borrow a basin of attraction argument from~\cite{Jadbabaie} to conclude that the finite-$n$ system achieves phase synchronization with high probability as $n$ tends to infinity.

Unlike previous works \cite{Abdalla,Kassabov}, our framework does \emph{not} rely on the gradient structure of the Kuramoto model; hence, it may provide insight into synchronization for a broader class of oscillator networks. To illustrate this advantage, we consider the Sakaguchi-Kuramoto model, which is not a gradient system. Here, we set the graphon equal to $p \in (0,1]$ and $\alpha_n=1$. In this case, we show that if the continuum model phase synchronizes, then the oscillators in the sampled network attain phases that lie within a $\frac{\pi}{2} - \beta$ distance of one another at a fixed time, with high probability. Using an argument from \cite{Bullo}, we then infer that under these same initial conditions, the Sakaguchi-Kuramoto model on the sampled network achieves frequency synchronization with high probability as $n \to \infty$, for fixed $p \in (0,1]$. 


\subsection{Roadmap}

The background section (Section \ref{Background}) consists of four parts. In Section \ref{W-Random Graph Model}, we introduce the mathematical notion of a graphon and explain how to generate finite $W$-random networks from it. Section \ref{Infinite Population Oscillator Dynamics} introduces the continuum dynamical system. In Section \ref{Sampled Oscillator Dynamics}, we define the sampled  dynamical system that interacts according to the $W$-random network obtained from the graphon. In Section \ref{Main Convergence Result Section}, we state our main result: For $n$ sufficiently large, a piecewise interpolant of the solution to the sampled dynamics over a $W$-random network of size $n$ converges to the solution of the continuum dynamics in the $L^{\infty}$ norm for any fixed time with high probability. 

The proof of our main convergence result is given in Section \ref{Convergence Proofs} and involves two intermediate stages. First, for large enough $n$, we show that the solution of the sampled dynamics converges to the solution of a simpler system that we call the \textit{averaged oscillator dynamics}, in which the oscillators are assumed to interact over a complete graph instead of a random graph (Section \ref{Comparing the Sampled System to the Averaged System}).  Section \ref{Comparing the Averaged System to the Continuum System} presents the second stage of the proof. There we show that a piecewise interpolant of the solution to the averaged dynamics converges to the solution of the continuum dynamics. {This is where regularity assumptions on the graphon are needed.} Finally, in Section \ref{Comparing the Sampled System to the Continuum System} we combine the results from Section \ref{Comparing the Sampled System to the Averaged System} and Section \ref{Comparing the Averaged System to the Continuum System} to prove our main convergence result. 

In Section~\ref{Synchronization}, we apply our convergence result to study phase synchronization in Kuramoto oscillators and frequency synchronization in Sakaguchi-Kuramoto oscillators, both interacting over large Erd\H{o}s--R\'enyi random networks. This is done by considering the behavior of the corresponding continuum models, leveraging basin of attraction arguments, and using the convergence result to transfer insights from the continuum to the large, finite network setting.

\section{Background}\label{Background}
\subsection{\texorpdfstring{$W$}{W}-Random Graph Model}\label{W-Random Graph Model}

Let $I$ denote the closed unit interval $[0,1]$ and let $W:I^{2}\rightarrow I$ be a continuous, real-valued, symmetric function that we refer to as a graphon. 

As discussed earlier, graphons have two interpretations. First, graphons may be used to describe interactions among an infinite population (with nodes indexed by real numbers on the interval $[0,1]$) and second, graphons may be interpreted as a discrete random graph model that gives rise to a $W$-random graph generation process. 

Figure~\ref{fig1} illustrates how a graphon $W$ can be discretized and then used to obtain a random network. From $W$, we construct\footnote{Note that this $W$-random graph generation process differs slightly from the conventional generative process described in \cite{Lovasz} and \cite{Borgs}, in which $W^{(n)}_{ij}$ is obtained by evaluating the graphon at a pair of independently sampled points from the unit interval rather than an averaging scheme.} a sampled, undirected $n\times{n}$ $W$-random network with adjacency matrix $A^{(n)}$ such that $A_{ij}^{(n)} = A_{ji}^{(n)}= \textrm{Ber}(\alpha_n W^{(n)}_{ij})$ where
\begin{align}\label{eqn:graphon-sampling}
    W^{(n)}_{ij} = n^{2}\int_{I_i^{(n)}\times{I_j^{(n)}}} W(x,y) d x d y.
\end{align} 

\noindent Here, $\text{Ber}(p)$ denotes a Bernoulli random variable with probability $p$ and $I_i^{(n)} = \left[\frac{i-1}{n}, \frac{i}{n} \right)$ where $1\leq {i}\leq{n}$, and the \textit{scaling factor}, $\alpha_n$ is such that $\alpha_n\leq{1}$.


\begin{figure}[t]
\centerline{\includegraphics[width=\columnwidth]{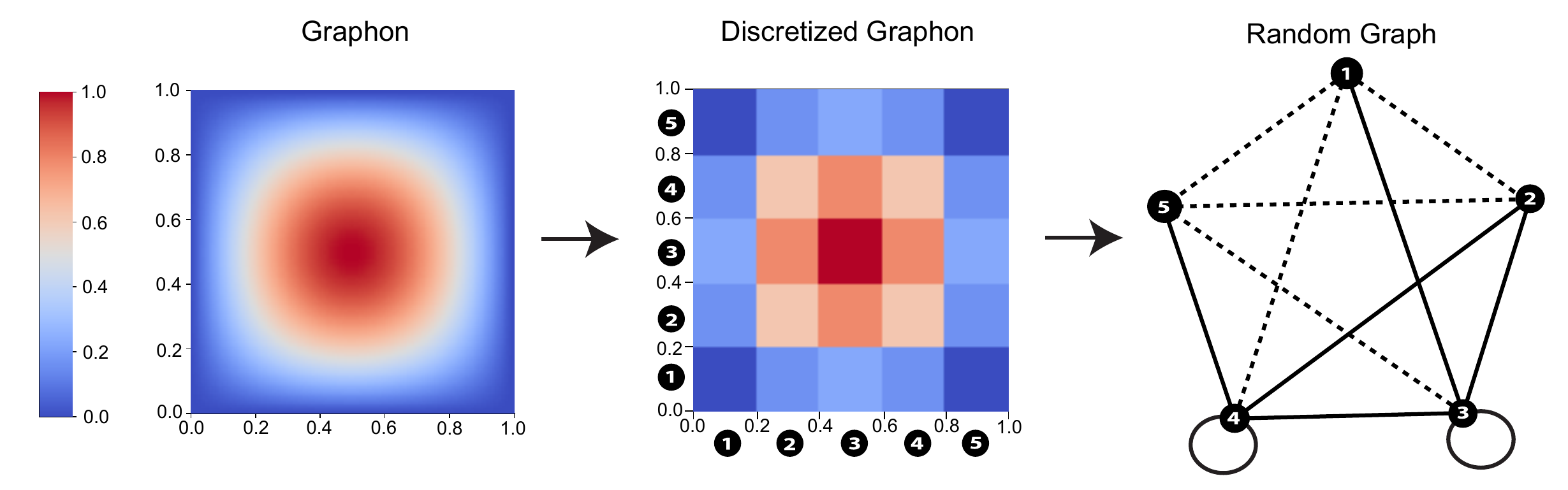}}
\caption{Schematic illustration of how to obtain a random network from a graphon. In the example shown, a random graph on $5$ nodes is generated from the continuous graphon $W(x,y) = \sin(\pi{x})\sin(\pi{y})$. The discretized graphon is a function defined on the unit square, $I^{2}$, where each $x,y\in {I_i^{(5)}\times{I_j^{(5)}}}$ assumes the value $W^{5}_{ij} = 5^{2}\int_{I_i^{(5)}\times{I_j^{(5)}}}W(x,y)dydx$. From this discretized graphon, one may construct a random network of size $5$ by letting the probability that there exists an edge between nodes $i$ and $j$ be equal to $W^{(5)}_{ij}$.}
\label{fig1}
\end{figure}

\subsection{Infinite Population Oscillator Dynamics}\label{Infinite Population Oscillator Dynamics}
With respect to the first interpretation of graphons, we consider an infinite population of oscillators interacting according to the following continuum dynamical system (CDS):


\begin{equation}\label{eqn:continuum_dynamics}
  \partial_t \theta(t, x)=f(\theta(t,x))+\int_I W(x, y) D(\theta(t, y)-\theta(t, x)) d y, \quad x \in I\tag{CDS}
\end{equation}
where the initial condition $\theta(0, x)=\eta(x)$ and $\eta\in C([0,1])$. In other words, we assume the initial phases of the oscillators vary smoothly with the oscillator index $x$, as given by a continuous function $\eta(x).$ Furthermore, we assume that the oscillators' uncoupled dynamics $f(\theta)$ is independent of $x$. This is what we mean by saying that the oscillators are identical. We assume that $f(\theta)$ is Lipschitz continuous (with Lipschitz constant $L_f$), $2\pi$-periodic in $\theta$. The coupling kernel $D$ is a $2 \pi$-periodic Lipschitz continuous function (with Lipschitz constant $L_D$) and such that $\max _{\theta\in \mathbb{R} }|D(\theta)|=1$. In Theorem~\ref{thm:existence} stated below, we show the existence of a unique, global-in-time solution $\theta(t,\cdot)\in C([0,1])$ of the continuum system~\eqref{eqn:continuum_dynamics}. Note that similar results may be found in~\cite{Li} or~\cite{Kaliuzhnyi-Medvedev}, but we include a proof in the Appendix for completeness. 

\begin{theorem}\label{thm:existence}
    Let $\eta\in C([0,1],\mathbb{R})$ and $W\in C([0,1]^2,\mathbb{R})$.
    \begin{enumerate}
        \item[(i)] For any $T>0$, the system~\eqref{eqn:continuum_dynamics} has a unique solution $\theta\in ([0,T];L^\infty([0,1]))$.
        \item[(ii)] Furthermore, $\theta\in C([0,T]\times[0,1];\mathbb{R})$ and $\partial_t\theta(t,x)\in C([0,T]\times[0,1];\mathbb{R})$.
    \end{enumerate}
\end{theorem}
\begin{proof}
    Refer to Appendix~\ref{appendix:existence-theorem}.
\end{proof}

\subsection{Sampled Oscillator Dynamics}\label{Sampled Oscillator Dynamics}
\sloppy
Adopting the perspective from \cite{Lovasz} in which we view a graphon as a random graph model, for each oscillator $i\in[n]:=\{1,...,n\}$ we consider the following sampled dynamical system (SDS):

\fussy
\begin{equation}\label{eqn:sampled_dynamics}
  \dot{\theta}_i^n(t)=f\left(\theta_i^n(t)\right)+\frac{1}{n\alpha_n} \sum_{j=1}^n A_{i j}^{(n)} D\left(\theta_j^n(t)-\theta_i^n(t)\right)\tag{SDS}
\end{equation}
with initial conditions ${\theta}_i^n(0) = \eta\left(\frac{i-1}{n}\right)$. Here, $\theta_i^n: [0,\infty) \rightarrow \mathbb{R}$ is the phase of oscillator $i$ as a function of time and $A_{i j}^{(n)}$ is the sampled adjacency matrix obtained from the graphon. To compare solutions of \eqref{eqn:sampled_dynamics}, $\theta^{n} = (\theta^{n}_1,\dots,\theta^{n}_n)^{T}$, to solutions of \eqref{eqn:continuum_dynamics} we define a piecewise constant interpolant
$$
    \theta^{n}(t,x) = \sum_{i=1}^{n} \theta^{n}_i(t) \mathbbm{1}_{\left[\frac{i-1}{n},\frac{i}{n}\right)} (x).
$$

\noindent Here, $\mathbbm{1}_E$ denotes the characteristic function of a set $E$, defined by
\begin{align*}
    \mathbbm{1}_E(x):=\begin{cases}
        1 &\text{if }x\in E,\\
        0 &\text{if }x\notin E.
    \end{cases}
\end{align*}

\noindent We define the $L^\infty$ norm of a (piecewise) continuous function $u$ on the interval $I$ as
\begin{align*}
    \|u\|_{L^\infty(I)}=\max_{x\in[0,1]}|u(x)|.
\end{align*}

\subsection{Main Convergence Result}\label{Main Convergence Result Section}

\begin{theorem} \label{main_theorem}
Suppose that $W\in C(I^2)$ is a symmetric function, $\frac{\log (n) / n}{\alpha_n^3} \rightarrow 0$, and let ${\theta}^{}(t,x)$ be the solution to (CDS) and ${\theta}^{n}(t,x)$ be the piecewise constant interpolant solution of (SDS) with initial condition $\eta\in C([0,1])$. For any fixed $\delta,\epsilon,T,>0$, there exists $\bar{n}\in\mathbb{N}$  such that for each $n>{\bar{n}},$ 
\[
    \|{\theta}^{n}(t,x)-{\theta}^{}(t,x)\|_{L^{\infty}(I)} < \epsilon, \quad \forall{t}\leq{T}
\] 
with probability at least $1-\delta$.
\end{theorem}

\section{Convergence Proofs} \label{Convergence Proofs}

\subsection{Comparing the Sampled System to the Averaged System}\label{Comparing the Sampled System to the Averaged System}

To prove our main convergence result (Theorem \ref{main_theorem}), we first compare the solutions of \eqref{eqn:sampled_dynamics} to the solutions of an averaged dynamical system (ADS) where each oscillator $i\in[n]$ adopts the following dynamics:

\begin{equation}\label{averaged_dynamics}
  \dot{\bar{\theta}}_i^n(t)=f\left(\bar{\theta}_i^n (t)\right)+\frac{1}{n} \sum_{j=1}^n W_{i j}^{(n)} D\left(\bar{\theta}_j^n (t)-\bar{\theta}_i^n(t)\right)\tag{ADS}
\end{equation}

\noindent where ${\bar{\theta}}_i^n(0) = \eta\left(\frac{i-1}{n}\right)$. Note the only difference between \eqref{eqn:sampled_dynamics} and \eqref{averaged_dynamics} is that $W_{i j}^{(n)}$ is used in \eqref{averaged_dynamics} while $\frac{A_{i j}^{(n)}}{\alpha_n}$ is used in \eqref{eqn:sampled_dynamics}. Again, $\bar{\theta}_i^n:[0,\infty) \rightarrow \mathbb{R} $ is the phase of oscillator $i$ as a function of time and $\bar{\theta}^{n} = (\bar{\theta}^{n}_1,\dots,\bar{\theta}^{n}_n)^{T}$. 

\begin{proposition}\label{convgence_sampled_avged}
Suppose that $W\in C(I^2)$ is a symmetric function, $\frac{\log (n) / n}{\alpha_n^3} \rightarrow 0$, and let ${\theta}^{n}(t)$ be the solution to (SDS) and ${\bar{\theta}}^{n}(t)$ be the solution of (ADS) with initial condition $\eta\in C([0,1])$. For any fixed $\delta,\epsilon,T>0$, there exists $n_1\in\mathbb{N}$ such that for each $n>{n_1},$ 
$$
    \|{\theta}^{n}(t)-\bar{\theta}^{n}(t)\|_{\infty} < \epsilon,  \quad \forall{t\leq{T}}
$$ 
\noindent with probability at least $1-\delta$.
\end{proposition}

\begin{proof} Define the variable $\phi_i^{n}(t) = {\theta}_i^{n}(t)- \bar{\theta}_i^{n}(t)$ and $u_{i}(t) = (\phi_i^{n}(t))^{2}$. Observe that for all $t$,

\begin{align*}
\dot{u}_i(t) =& \underbrace{2 \phi_i^n(t) \left(f\left(\theta_i^n(t)\right)-f\left(\bar{\theta}_i^n(t)\right)\right)}_{I_1(t)}\\
&+ \underbrace{2 \phi_i^n(t)\left(\frac{1}{ n\alpha_n} \sum_{j=1}^n A^{(n)}_{ij}D\left(\bar{\theta}_j^n(t)-\bar{\theta}_i^n(t)\right)-\frac{1}{n} \sum_{j=1}^n W^{(n)}_{ij} D\left(\bar{\theta}_j^n(t)-\bar{\theta}_i^n(t)\right)\right)}_{I_2(t)}\\
&+\underbrace{2 \phi_i^n(t)\left(\frac{1}{n\alpha_n} \sum_{j=1}^n A_{i j}^{(n)} D\left(\theta_j^n(t)-\theta_i^n(t)\right)-\frac{1}{ n\alpha_n} \sum_{j=1}^n A_{i j}^{(n)} D\left(\bar{\theta}_j^n(t)-\bar{\theta}_i^n(t)\right)\right)}_{I_3(t)}.
\end{align*}

\noindent By the triangle inequality we have that $\dot{u}_i \leq{|I_1(t)|+|I_2(t)|+|I_3(t)|}.$ We start by bounding $|I_1(t)|$: 

\begin{align*}
|I_1(t)| \leq 2\left| \phi_i^n (t)\right| L_{f}\left|\phi_{i}^{n}(t) \right| = {2 L_{f} (\phi_i^n(t))^{2} \label{eqn:sampled-avgd-termI1}=  2 L_{f} u_{i}(t).}
\end{align*}

\noindent Using the fact that $ ab \leq \frac{1}{2}a^2 + \frac{1}{2}b^2 $ for all $ a, b \in \mathbb{R}$, we can bound $|I_2(t)|$:

\begin{align*}
&|I_2(t)| \leq \left|2 \phi_i^n(t)\right| \cdot\left|\frac{1}{n} \sum_{j=1}^n {\left(\frac{A^{(n)}_{i j}}{\alpha_n}-W^{(n)}_{i j}\right)} D\left(\bar{\theta}_j^n(t)-\bar{\theta}_i^n(t)\right)\right|\\
&\leq \frac{1}{2}\left(2 \phi_i^n(t)\right)^2 + \frac{1}{2}\left(\underbrace{\frac{1}{n} \sum_{j=1}^n {\left(\frac{A^{(n)}_{i j}}{\alpha_n}- W^{(n)}_{i j}\right)} {D\left(\bar{\theta}_j^n(t)-\bar{\theta}_i^n(t)\right)}}_{=:g_{n,i}(t)}\right)^2  = 2u_i(t)+\frac{1}{2}(g_{n,i}(t))^{2}.
\end{align*}

\noindent Bounding $|I_3(t)|$:

\begin{align*}
&|I_3(t)| \leq \left|2 \phi_i^n(t)\right|\left|\frac{1}{n} \sum_{j=1}^{n} \frac{A^{(n)}_{i j}}{\alpha_n}\left(D\left(\theta_j^n(t)-\theta_i^n(t)\right)-D\left(\bar{\theta}_j^n(t)-\bar{\theta}_i^n(t)\right)\right)\right|\\
&\leq \left|2 \phi_i^n(t)\right| \frac{1}{n } \sum_{j=1}^{n} \frac{A^{(n)}_{i j}}{\alpha_n} L_D\left(\left|\phi_j^n(t)\right|+\left|\phi_i^n(t)\right|\right)\\
&=  \frac{L_{D}}{n } \sum_{j=1}^{n} \frac{A^{(n)}_{i j}}{\alpha_n}\left(2\left|\phi_i^n(t)\right|\left|\phi_j^n(t)\right|+2\left|\phi_i^n(t)\right|\left|\phi_i^n(t)\right|\right)\\
&\leq  \frac{L_{D}}{n } \sum_{j=1}^{n} \frac{A^{(n)}_{i j}}{\alpha_n} \left((\phi_i^n(t))^{2}+(\phi_j^n(t))^{2}+2u_i(t)\right)\\
&= \frac{L_{D}}{n } \sum_{j=1}^{n} \frac{A^{(n)}_{i j}}{\alpha_n} \left(3u_i(t)+u_j(t)\right)\\
&= L_{D} \left( \frac{3u_i(t)}{n}\sum_{j=1}^{n} \frac{A^{(n)}_{i j}}{\alpha_n} + \frac{1}{n}\sum_{j=1}^{n} \frac{A^{(n)}_{i j}}{\alpha_n} u_j(t)\right).\\
\end{align*}

 By Lemma \ref{Upper Bounding Sampled Adjacency Matrix Part 2} , there exists a $n_{1a}$ such that for all $n\geq{n_{1a}}$ $\frac{1}{n}\sum_{j=1}^{n} \frac{A^{(n)}_{i j}}{\alpha_n}\leq 2$ for all $i$ with probability $1-\frac{\delta}{2}$. Hence, with probability $1-\frac{\delta}{2}$,

$$\dot{u}_i(t) \leq \ \underbrace{(2 L_f+6 L_D+2)}_{=:c} u_i(t)+\frac{L_D}{\alpha_n n} \sum_{j=1}^{n} A^{(n)}_{i j} u_j(t)+\frac{1}{2} (g_{n,i}(t))^2$$

\noindent where $u_i(0) = (\phi_i^{n}(0))^{2} = ({\theta}_i^{n}(0)- \bar{\theta}_i^{n}(0))^{2} = 0$. Define
$$\dot{s}_i(t) = {c} s_i(t)+\frac{L_D}{\alpha_n n} \sum_{j=1}^{n} A^{(n)}_{i j} s_j(t)+\frac{1}{2} (g_{n,i}(t))^2$$

\noindent where $s_i(0) = 0$. 

\vskip0.1in

Two matrices \( X, Y \) satisfy the elementwise inequality \( X \leq_e Y \) if
\[
X_{ij} \leq Y_{ij} \quad \text{for all } i, j.
\]

\vskip0.1in

Let \( B = c I + \frac{L_D}{\alpha_n n} A^{(n)} \geq_e \mathbf{0} \), where \( \mathbf{0} \) is the zero matrix of appropriate size. Define the vectors ${\omega}(t) = [\omega_1(t), \dots, \omega_n(t)]^{T}, {u}(t) = [u_1(t), \dots, u_n(t)]^{T}$, and ${s}(t) = [s_1(t), \dots, s_n(t)]^{T}$, where \( \omega_i(t) = \frac{1}{2}(g_{n,i}(t))^2 \geq 0 \). Observe, ${\dot{{u}}}(t) \leq_{e} B u(t)+\omega(t)$ and ${\dot{{s}}}(t) = B s(t)+\omega(t)$. Then, by Lemma \ref{Positive Dynamical System}, ${u}(t) \leq_e {s}(t)$ for all $t$. 

Observe that $B$ is a Metzler matrix (the off-diagonal terms are non-negative) and $\omega(t)\geq_{e} \mathbf{0}$ for all $t$, implying that $s(t)\geq_{e}\mathbf{0}$  for all $t$. Hence,  $\dot{s}(t)\geq_{e}\mathbf{0}$ for all $t$ which means that $s(t)$ is non-decreasing, and 

$$s(t)\leq_e s(T)=\int_0^T e^{B(T-\tau)} \omega(\tau) d \tau \quad {\forall{t\leq{T.}}}$$ 

By Lemma \ref{Monotonicity of Matrix Powers and Exponential},

$$0 \leq_{e} B(T-\tau) \leq_{e} B \cdot T \Rightarrow e^{B(T-\tau)} \leq_{e} e^{B T}.$$

\noindent Hence, 

$$\int_0^T e^{B(T-\tau)} \omega(\tau) d \tau \leq_{e} e^{BT}\int_0^T  \omega(\tau) d \tau  = \frac{e^{BT}}{2}\left[\int_0^T\left(g_{n,i}(\tau)\right)^2 d \tau\right]_{i=1}^{n}.$$

\noindent By Lemma \ref{Upper Bound over Finite Time Interval on Sampled Network}, for any $\bar{\epsilon}$ there exists $n_{1b}$ such that for all $n\geq n_{1b}$, $\int_0^T\left(g_{n,i}(\tau)\right)^2 d \tau \leq \bar{\epsilon}$ for all $i$ with probability at least $1-\frac{\delta}{2}$. By the union bound, for any $n\geq n_1 :=\max\{n_{1a},n_{1b}\}$, with probability $1-\delta$,

\begin{align*}
&u(t) \leq_e s(t) \leq_e s(T) \leq_e \frac{\bar{\epsilon}}{2} e^{B T} \cdot \mathbf{1} \\
&\implies u_i(t) \leq \frac{\bar{\epsilon}}{2} \max_k\left(\left[e^{B T} \cdot \mathbf{1}\right]_k\right) = \frac{\bar{\epsilon}}{2} \| e^{B T}\|_{\infty} \leq \frac{\bar{\epsilon}}{2} e^{\| {B}\|_{\infty}{T}} \quad \forall{t\leq{T}}.\\
\end{align*}

\noindent The last inequality holds by Lemma \ref{Infinity Norm Bounds for Nonnegative Matrices}, since $B\geq_{e}0$. Finally, note that with the same probability 

$$
\|B\|_{\infty} \leq c+L_{D}\left\|\frac{A^{(n)}}{\alpha_n  n}\right\|_{\infty} \leq c+L_D \max_i \frac{1}{n} \sum_{j=1}^{n} \frac{A^{(n)}_{i j}}{\alpha_n }\leq c+2L_{D}, 
$$
\noindent where the last inequality holds because we have already shown $\frac{1}{n}\sum_{j=1}^{n} \frac{A^{(n)}_{i j}}{\alpha_n}\leq 2$ for all $i$. Thus, for all $i$, $u_i(t) \leq \frac{\bar{\epsilon}}{2} e^{\left(c+2 L_D\right) T}$ for all $t\leq{T}$. To conclude, set $\bar{\epsilon}=\frac{2 \epsilon }{ e^{\left(c+2 L_D\right) T}}$.
\end{proof}

\subsection{Comparing the Averaged System to the Continuum System}\label{Comparing the Averaged System to the Continuum System}

From Theorem~\ref{thm:existence}, we have the existence of a unique, global-in-time solution $\theta(t,x)$ of the continuum system~\eqref{eqn:continuum_dynamics}. We now prove that this solution remains $L^\infty$ close to the solutions of the averaged system~\eqref{averaged_dynamics}. To compare solutions of \eqref{averaged_dynamics}, $\bar{\theta}^{n} = (\bar{\theta}^{n}_1,\dots,\bar{\theta}^{n}_n)^{T}$, to solutions of \eqref{eqn:continuum_dynamics} we define a piecewise constant interpolant

\begin{align}\label{eq:averaged-solution-piecewise-interpolant}
    \bar{\theta}^{n}(t,x) = \sum_{i=1}^{n} \bar{\theta}^{n}_i(t) \mathbbm{1}_{\left[\frac{i-1}{n},\frac{i}{n}\right)} (x).
\end{align}

\begin{proposition}\label{convgence_avged_continuum}
Suppose that $W\in C(I^2)$ is a symmetric function. Let ${\theta}^{}(t,x)$ be the solution to (CDS), with initial condition $\eta\in C([0,1])$, and ${\bar{\theta}}^{n}(t,x)$ be the piecewise constant interpolant solution of (ADS). For any fixed $\epsilon,T>0$, there exists $n_2\in\mathbb{N}$ such that for all $n>n_2$,
\begin{equation}
    \|\bar{\theta}^{n}(t,x)-\theta(t,x)\|_{L^\infty(I)}< \epsilon, \quad \forall{t\leq{T}}.
\end{equation} 
\end{proposition}

\begin{proof}
\noindent
Fix $T>0$ and let $W^{(n)}(x,y):=\sum_{i,j=1}^{n} W_{ij}^{(n)}\mathbbm{1}_{I_i^{(n)}\times{I_j^{(n)}}}(x,y)$ and consider the difference between the averaged model and the continuum model in the ${L^\infty}$ norm.\footnote{Recall that by standard ODE theory, $\bar{\theta}^{n}_i(t)$ is continuously differentiable in time~\cite{arnold1992ordinary} implying that the following is well defined:
\begin{align*}
    \partial_t\bar{\theta}^n(t,x)=\sum_{i=1}^{n} \dot{\bar{\theta}}^{n}_i(t) \mathbbm{1}_{\left[\frac{i-1}{n},\frac{i}{n}\right)} (x). 
\end{align*} 
}
Note that, by Theorem \ref{thm:existence}, $\partial_t \bar{\theta}^{n}(t,x)-\partial_t \theta(t,x) $ is piece-wise continuous in $x$, thus: 

\begin{align*}
\left\|\partial_t \bar{\theta}^{n}(t,x)-\partial_t \theta(t,x)\right\|_{L^\infty(I)} 
\leq& \left\|f(\bar{\theta}^{n}(t,x))- f({\theta}^{}(t,x)) \right\|_{L^\infty(I)}\\
&+\left\|\int_0^{1} W^{(n)}(x,y)D(\bar{\theta}^{n}(t,y)-\bar{\theta}^{n}(t,x))dy\right. \\
&- \left.\int_0^{1} W^{}(x,y)D({\theta}^{}(t,y)-{\theta}^{}(t,x))dy \right\|_{L^\infty(I)}.
\end{align*}

Since $f$ is assumed to be Lipschitz, we have 
\begin{align*}
\left\|f(\bar{\theta}^{n}(t,x))- f({\theta}^{}(t,x)) \right\|_{L^\infty(I)} \leq L_f \left\|\bar{\theta}^{n}(t,x)- {\theta}^{}(t,x)\right\|_{L^\infty(I)}.
\end{align*}

Moreover, 
\begin{align*}
\left\|\int_0^{1}\right. &W^{(n)}(x,y)D(\bar{\theta}^{n}(t,y)-\bar{\theta}^{n}(t,x))dy - \left.\int_0^{1} W^{}(x,y)D({\theta}^{}(t,y)-{\theta}^{}(t,x))dy \right\|_{L^\infty(I)}\\ 
\leq& \underbrace{\max_{x\in[0, 1]} \int_0^{1} | W^{(n)}(x,y)D(\bar{\theta}^{n}(t,y)-\bar{\theta}^{n}(t,x)) - W^{}(x,y)D(\bar{\theta}^{n}(t,y)-\bar{\theta}^{n}(t,x))|dy}_{I_1(t)}\\ 
&+ \underbrace{\max_{x\in[0, 1]} \int_0^{1}|W^{}(x,y)D(\bar{\theta}^{n}(t,y)-\bar{\theta}^{n}(t,x)) - W^{}(x,y)D({\theta}^{}(t,y)-{\theta}^{}(t,x))| dy}_{I_2(t)}.
\end{align*}
Recalling that $|D|\leq 1$ we have
\begin{align}\label{eqn:avgd-cont-termI1}
|I_1(t)|\leq \max _{x\in[0,1]} \int_0^1\left|W^{(n)}(x, y)-W(x, y)\right| d y \leq \left\|W^{(n)}(x,y)-W^{}(x,y)\right\|_{L^\infty(I^2)}.
\end{align}
On the other hand, since $|W|\leq 1$ and $D$ is Lipschitz,
\begin{align}\label{eqn:avgd-cont-termI2}
|I_2(t)|\leq \max_{x\in[0,1]}L_{D} \int_{0}^1\left|\left(\bar{\theta}^n(t, y)- \theta(t, y)\right)-\left(\bar{\theta}^n(t, x)- \theta(t, x)\right)\right|dy \leq 2L_{D} \left\|\bar{\theta}^{n}(t,x)-{\theta}^{}(t,x)\right\|_{L^\infty(I)}.
\end{align}
Putting~\eqref{eqn:avgd-cont-termI1} and~\eqref{eqn:avgd-cont-termI2} together, we obtain the following estimate:
\begin{align*}
\left\|\partial_t\bar{\theta}^{n}(t,x)-\partial_t\theta(t,x)\right\|_{L^\infty(I)}\leq& (L_f+{2}L_{D}) \left\|\bar{\theta}^{n}(t,x)- {\theta}^{}(t,x)\right\|_{L^\infty(I)}\\
&+ \left\|W^{(n)}(x,y)-W^{}(x,y)\right\|_{L^\infty(I^2)}.  
\end{align*}
By the Fundamental Theorem of Calculus, for all $t>0$
\begin{align*}
    \bar{\theta}^n(t,x)-\theta(t,x) = \bar{\theta}^n(0,x)-\theta(0,x)+\int_0^t\partial_t\bar{\theta}^n(\tau,x)-\partial_t\theta(\tau,x) d\tau.
\end{align*}

Taking the $L^\infty$ norm in the above equation we get
\begin{align*}
\left\|\bar{\theta}^{n}(t,x)- {\theta}^{}(t,x)\right\|_{L^\infty(I)}
\leq& \left\|{\bar{\theta}}^{n}(0,x)- {\theta}^{}(0,x) \right\|_{L^\infty(I)}+ \int_0^{t} \left\|\partial_t \bar{\theta}^{n}(\tau,x)-\partial_t\theta(\tau,x)\right\|_{L^\infty(I)}d\tau\\
\leq&\left\|{\eta}_{n}(x)- {\eta}^{}(x) \right\|_{L^\infty(I)} +\left\|W^{(n)}(x,y)-W^{}(x,y)\right\|_{L^\infty(I^2)}t\\ 
&+ (L_f+{2}L_{D}) \int_0^{t}\left\|\bar{\theta}^{n}(\tau,x)- {\theta}^{}(\tau,x)\right\|_{L^\infty(I)} d\tau,
\end{align*}
where $\eta_n$ is defined as the following discretization:
\begin{align*}
    \eta_{n}(x) := \sum_{i=1}^{n} \eta\left(\frac{i-1}{n}\right) \mathbbm{1}_{\left[\frac{i-1}{n},\frac{i}{n}\right)} (x).
\end{align*}
Note that $f(t) = \left\|\bar{\theta}^{n}(t,x)- {\theta}^{}(t,x)\right\|_{L^\infty(I)}$ is continuous in $t$ by Lemma \ref{lemma:difference-of-solutions-continuity}. By Gronwall's inequality~\cite{hirsch2012differential},
\begin{align*}
&\left\|\bar{\theta}^{n}(t,x)- {\theta}^{}(t,x)\right\|_{L^\infty(I)} \leq (A + Bt) + L \int_0^{t} \left(A + B\tau\right)e^{L(t-\tau)} d\tau,
\end{align*}
where $A:=\|\eta_n(x)-\eta(x)\|_{L^\infty(I)}$, $B:=\left\|W^{(n)}(x,y)-W^{}(x,y)\right\|_{L^\infty(I^2)}$ and $L:=L_f+{2}L_{D}$.
Since $t>0$ and using integration by parts, we have
\begin{align}
\label{eqn:avgd-cont-estimate}
\left\|\bar{\theta}^{n}(t,x)- {\theta}^{}(t,x)\right\|_{L^\infty(I)} \leq Ae^{Lt}-\frac{B}{L}+\frac{B}{L}e^{Lt}\leq\left(A+\frac{B}{L}\right)e^{Lt}.
\end{align}
Since $\eta$ is continuous on a compact interval, it is uniformly continuous on $[0,1]$. Thus, it admits a modulus of continuity, $\gamma_1:[0,\infty)\rightarrow[0,\infty)$, which is increasing and satisfies $\lim_{s\to0}\gamma_1(s)=0$, such that $|\eta(x)-\eta(y)|\leq\gamma_1(|x-y|)$. 
This means that we can bound $A$ as follows
\begin{align}\label{eqn:A-estimate}
    A&=\max_{i\in [n]}\max_{x\in[\frac{i-1}{n},\frac{i}{n}]}\left|\eta(x)-\eta\left(\frac{i-1}{n}\right)\right|\leq\max_{i\in[n]}\max_{x\in[\frac{i-1}{n},\frac{i}{n}]}\gamma_1\left(\left|x-\frac{i-1}{n}\right|\right)\nonumber\\
    &\leq\max_{i\in[n]}\gamma_1\left(1/n\right)=\gamma_1(1/n)
\end{align}
where $\gamma_1(1/n)\rightarrow 0$ as $n\rightarrow\infty$.
\smallskip

Focusing now on $B$,  
\begin{align*}
    B =\max_{i,j\in{[n]\times[n]}}\max_{x,y\in{I_i^{(n)}\times I_j^{(n)}}} \left| W^{(n)}_{ij} -W(x,y) \right|,
\end{align*}
where $W^{(n)}_{ij}$ is defined by~\eqref{eqn:graphon-sampling}. Moreover,
$$
\min_{x,y\in{{I_i^{(n)}}\times{I_j^{(n)}}}} W(x,y) \leq{W^{(n)}_{ij}}\leq \max_{x,y\in{{I_i^{(n)}}\times{I_j^{(n)}}}} W(x,y).
$$
Since $W(x,y)$ is continuous, by the intermediate value theorem, there exists $(x_{ij},y_{ij})\in[I_i^{(n)}\times{I_j^{(n)}}]$ such that $W^{(n)}_{ij} = W(x_{ij},y_{ij})$. Thus, 
\begin{align*}
    &B = \max_{i,j\in{[n]\times[n]}}\max_{x,y\in{I_i^{(n)}\times I_j^{(n)}}} \left|W(x,y)-W(x_{ij},y_{ij}) \right|.
\end{align*}


\noindent{Since $W$ is continuous on a compact set, it is uniformly continuous on $[0,1]^2$. Thus, it admits a modulus of continuity, $\gamma_2:[0,\infty)\rightarrow[0,\infty)$, which is increasing and satisfies $\lim_{s\rightarrow0}\gamma_2(s)=0$, such that
\begin{align*}
    \max_{x,y\in {I_i^{(n)}\times I_j^{(n)}}}\left|W(x,y)-W(x_{ij},y_{ij})\right| \leq \max_{x,y\in {I_i^{(n)}\times I_j^{(n)}}}  \gamma_2(|(x-x_{ij},y-y_{ij})|)\leq\gamma_2(\sqrt{2}/n).
\end{align*}}

\noindent Thus, we can bound $B$ as follows: 
\begin{align}
    \label{eqn:B-estimate}
    B = \left\|W^{(n)}(x,y)-W^{}(x,y)\right\|_{L^\infty(I^2)} \leq \gamma_2(\sqrt{2}/n),
\end{align}
where $\gamma_2(\sqrt{2}/n)\rightarrow0$ as $n\rightarrow\infty$.
Combining~\eqref{eqn:avgd-cont-estimate},~\eqref{eqn:A-estimate} and~\eqref{eqn:B-estimate} we get

\begin{align*}
&\left\|\bar{\theta}^{n}(t,x)- {\theta}^{}(t,x)\right\|_{L^\infty(I)} \leq (\gamma_1(1/n)+\gamma_2(\sqrt{2}/n)1/L)e^{Lt},\quad \forall{t\leq{T}}.
\end{align*}
Thus, for all $\epsilon>0$ there exists a $\bar{n}\in\mathbb{N}$ such that for all $n>\bar{n}$,
\begin{align*}
    \|\bar{\theta}^{n}(t,x)-\theta(t,x)\|_{L^\infty(I)}< \epsilon, \quad \forall{t\leq{T}}.
\end{align*}
\end{proof}

\subsection{Comparing the Sampled System to the Continuum System}\label{Comparing the Sampled System to the Continuum System}

\noindent Synthesizing Proposition~\ref{convgence_sampled_avged} and Proposition~\ref{convgence_avged_continuum}, we now prove Theorem~\ref{main_theorem}. 


\begin{proof}[Proof of Theorem~\ref{main_theorem}] Fix $\delta,T,\epsilon>0$ and suppose that $\frac{\log (n) / n}{\alpha_n^3} \rightarrow 0$. By Proposition~\ref{convgence_sampled_avged}, there exists $n_1\in\mathbb{N}$ such that for each $n>{n_1},$ 
$$
    \|{\theta}^{n}(t)-\bar{\theta}^{n}(t)\|_{\infty} = \|{\theta}^{n}(t,x)-\bar{\theta}^{n}(t,x)\|_{L^\infty(I)} < \frac{\epsilon}{2} \quad \forall{t\leq{T}}
$$
with probability at least $1-\delta$. By Proposition~\ref{convgence_avged_continuum}, there exists $n_2\in\mathbb{N}$ such that for all $n>n_2$,
$$
\|\bar{\theta}^{n}(t,x)-\theta(t,x)\|_{L^\infty(I)}< \frac{\epsilon}{2} \quad \forall{t\leq{T}}.
$$

\noindent
Thus, there exists $\bar{n} = \max\{n_1,n_2\}$ such that for all $n>\bar{n}$:
\begin{align*}
\|{\theta}^{n}(t,x)-{\theta}^{}(t,x)\|_{L^{\infty}(I)} &\leq \|{\theta}^{n}(t,x)-\bar{\theta}^{n}(t,x)\|_{L^\infty(I)} + \|\bar{\theta}^{n}(t,x)-\theta(t,x)\|_{L^\infty(I)}\\
&<\frac{\epsilon}{2} + \frac{\epsilon}{2}= \epsilon \quad \forall{t\leq{T}}
\end{align*}
with probability at least $1-\delta$.
\end{proof}

\section{Synchronization}\label{Synchronization}
In this section, we apply our main convergence result to study synchronization properties of identical interacting Kuramoto (\ref{phase_syncrhonization}) and Sakaguchi-Kuramoto (\ref{frequency_syncrhonization}) oscillators on Erd\H{o}s-R\'enyi random graphs. In particular, we aim at proving that synchronization in the continuum model implies synchronization in the sampled network with high probability. This is not obvious a-priori because the convergence result in Theorem~\ref{main_theorem}] holds for finite time $T$.

\subsection{Phase Synchronization in the Kuramoto Model} \label{phase_syncrhonization}

By setting $f=0$, $W=1$, and $D(\cdot)=\sin(\cdot)$ in~\eqref{eqn:continuum_dynamics} we obtain, as a special case, the dynamics of the \textit{homogeneous continuum Kuramoto model}

\begin{equation}\label{continuum_Kuramoto_model}
  \partial_t \theta(t, x)=\int_I \sin(\theta(t, y)-\theta(t, x)) d y, \quad x \in I.
\end{equation}

\noindent We aim at proving that if \eqref{continuum_Kuramoto_model} synchronizes as formalized in Assumption \ref{CKM_sync_assumption}, then synchronization also happens with high probability in the sampled network.  

\begin{assumption}\label{CKM_sync_assumption} 
Suppose that the initial condition for \eqref{continuum_Kuramoto_model}, $\theta(0, x)$, is such that 
\begin{itemize}
    \item $\theta(0, x)\in C([0,1])$, and 
    \item there exists a $c$ for which $\lim_{t\rightarrow{\infty}} \|\theta(t,\cdot)-c\|_{L^\infty(I)}=0.$
\end{itemize}
\end{assumption}

\begin{theorem}\label{Phase synch Kuramoto model on ER graphs}
Fix  $f=0$, $W=1$, $D(\cdot)=\sin(\cdot)$, and  $\frac{\log (n) / n}{\alpha_n^3} \rightarrow 0$. Suppose that $\theta(0,x)$ satisfies Assumption \ref{CKM_sync_assumption}. Let $\theta^{n}(0,x) = \sum_{i=1}^{n} \theta^{n}_i(0) \mathbbm{1}_{\left[\frac{i-1}{n},\frac{i}{n}\right]} (x)$ where ${\theta}_i^n(0) = \theta\left(0,\frac{i-1}{n}\right)$. Fix $\delta>0$. There exists some constant $\widetilde{c}$ and $\bar{n}\in\mathbb{N}$ such that for each $n>{\bar{n}}$, 
\begin{equation}\label{convergence_oscillatos_inf_norm}
\lim_{t\to\infty} \|{\theta}^{n}(t)-\widetilde{c}\|_{\infty} = 0
\end{equation}
\noindent with probability at least $1-\delta$. 
\end{theorem}

\begin{proof} 
By Assumption \ref{CKM_sync_assumption}, there exists a constant $c$ such that $\lim_{t\rightarrow{\infty}} \|\theta(t,\cdot)-c\|_{L^\infty(I)}=0$. By the definition of the limit, there exists a time $T>0$ such that for all $t\geq{T}$
$$
\|{\theta}^{}(t,x)-c\|_{L^{\infty}(I)} < \frac{\pi}{8}.
$$
\noindent In particular, $\|{\theta}^{}(T,x)-c\|_{L^{\infty}(I)} < \frac{\pi}{4}$. Fix $\delta>0$. By Theorem \ref{main_theorem},  for $\frac{\delta}{2}>0$, there exists $\bar{n}_1\in\mathbb{N}$ such that for each $n>{\bar{n}_1},$ 

$$
\|{\theta}^{n}(T,x)-{\theta}^{}(T,x)\|_{L^{\infty}(I)} < \frac{\pi}{8}
$$ 

\noindent
with probability at least $1-{\frac{\delta}{2}}$. By the triangle inequality, for each $n>{{\bar{n}_1}},$ 
\begin{align*}
\|{\theta}^{n}(T,x)-c\|_{L^{\infty}(I)} \leq \|{\theta}^{n}(T,x)-{\theta}^{}(T,x)\|_{L^{\infty}(I)} + \|{\theta}^{}(T,x)-c\|_{L^{\infty}(I)}
 < \frac{\pi}{8} + \frac{\pi}{8} = \frac{\pi}{4}
\end{align*}
\noindent

\noindent with probability at least $1-\frac{\delta}{2}$. 


For the given choice of $W$, the sampled network coincides with an Erd\H{o}s--R\'enyi random graph \( G(n, \alpha_n) \), where the edge probability exceeds the connectivity threshold \cite{erdos1960evolution}. As a result, the probability that the Erd\H{o}s--R\'enyi graph is connected converges to $1$ as \( n \to \infty \) \cite{erdos1960evolution}. Therefore, there exists \( \bar{n}_2 \in \mathbb{N} \) such that for all \( n > \bar{n}_2 \), the probability that the random graph is connected is at least \( 1 - \frac{\delta}{2} \).

Choose $\bar{n} = \max\left({\bar{n}_1},{\bar{n}_2 }\right)$ and let $n>\bar{n}$. By the union bound, the probability that the random graph is connected and $\|{\theta}^{n}(T,x)-c\|_{L^{\infty}(I)}<\frac{\pi}{2}$ is at least $1-(\frac{\delta}{2}+\frac{\delta}{2}) = 1-\delta$.  The conclusion follows, as established in the proof of Theorem 2 in \cite{Jadbabaie}, since all of the oscillators are within $\frac{\pi}{2}$ distance of each other and the graph is connected with probability at least $1-\delta$. 
\end{proof}

\subsection{Frequency Synchronization in the Sakaguchi-Kuramoto Model} \label{frequency_syncrhonization}
We next apply our main convergence result to study the Sakaguchi-Kuramoto model~\cite{sakaguchi1986solubl}. While the results in this section focus on studying \textit{frequency synchronization}, Figure~\ref{fig2} suggests that stronger results may be attainable.

\begin{figure}[H]
\centerline{\includegraphics[width=\columnwidth]{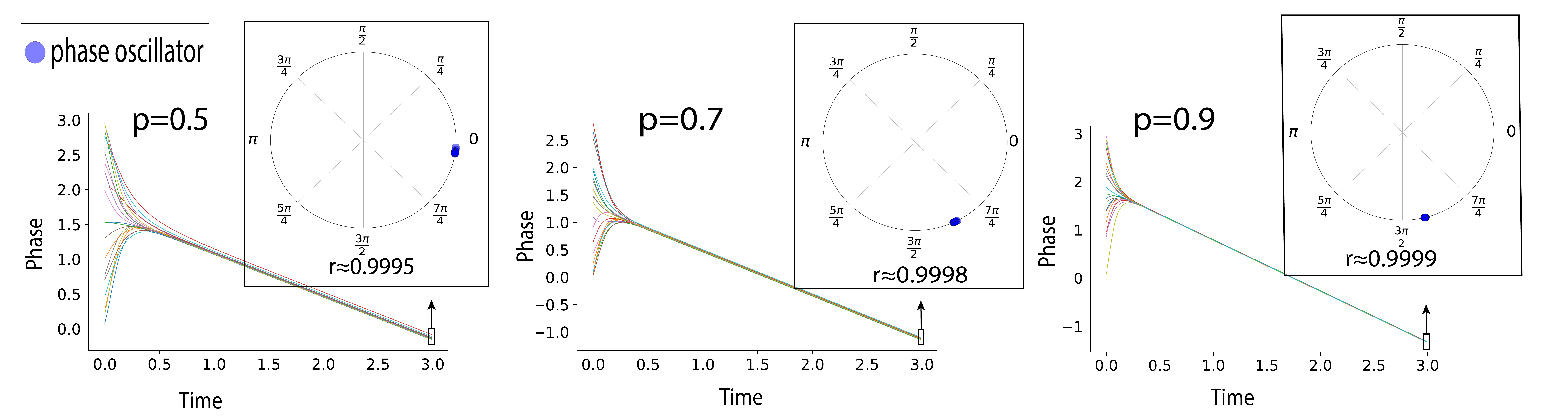}}
\caption{Snapshot of $n=20$ phase oscillators plotted on the unit circle after frequency synchronization has been achieved for Sakaguchi-Kuramoto oscillators (with $\beta=\frac{\pi}{50}$), interacting over an Erd\H{o}s-R\'enyi random network for three values of the edge probability: $p=0.5, p=0.7$, and $p=0.9$. In all three instances, the order parameter (a scalar $r\in[0,1]$ indicating how aligned the phases are \cite{Dorfler2014}) is close to $1$ and the oscillators are in nearly perfect phase alignment, suggesting a phenomenon beyond mere frequency synchronization.}
\label{fig2}
\end{figure}

Here, we set $f=0$, $W=p\in[0,1]$, and select $D(\cdot) = \sin(\cdot+\beta)$ to obtain the \textit{homogeneous continuum Sakaguchi-Kuramoto model}

\begin{equation}\label{continuum_Sakaguchi-Kuramoto_model}
  \partial_t \theta(t, x)=\int_I p \sin(\theta(t, y)-\theta(t, x)+\beta) d y, \quad x \in I
\end{equation}

\noindent
where $-\frac{\pi}{2}<\beta<\frac{\pi}{2}$. Again, we focus on initial conditions for which \eqref{continuum_Sakaguchi-Kuramoto_model} achieves phase synchronization.

\begin{assumption}\label{CSKM_sync_assumption} 
Suppose that the initial condition for \eqref{continuum_Sakaguchi-Kuramoto_model}, $\theta(0, x)$, is such that 
\begin{itemize}
    \item $\theta(0,x)\in C([0,1])$, and 
    \item there exists a $c$ for which $\lim_{t\rightarrow{\infty}} \|\theta(t,\cdot)-c\|_{L^\infty(I)}=0.$
\end{itemize}

\end{assumption}

\begin{theorem}\label{freq synch Sakaguchi Kuramoto model on ER graphs}
(Frequency Synchronization) Suppose $f=0$, $W=p\in[0,1]$, $D(\cdot) = \sin(\cdot+\beta)$ for $\beta\in(0,\frac{\pi}{2})$, and $\alpha_n=1$. Assume that $\theta(0,x)$ satisfies Assumption \ref{CSKM_sync_assumption}. Let $\theta^{n}(0,x) = \sum_{i=1}^{n} \theta^{n}_i(0) \mathbbm{1}_{\left[\frac{i-1}{n},\frac{i}{n}\right]} (x)$ where ${\theta}_i^n(0) = \theta\left(0,\frac{i-1}{n}\right)$. Fix $\delta>0$. There exists a $\bar{n}$ such that for all $n>\bar{n}$ and for any $\beta$ such that 
\begin{equation}\label{alpha_condition_for_sycnh}
\frac{\cos^{2} (\beta)}{\sin^{}(\beta)} > \frac{2}{p} \left(\frac{1+\frac{1}{n^{1/3}}}{1-\frac{1}{n^{1/3}}-\frac{2}{n}}\right),
\end{equation}
\noindent there exists some constant $\widetilde{c}$ such that  
\begin{equation}\label{freq_convergence_oscillatos_inf_norm}
\lim_{t\to\infty} \|{\dot{\theta}}^{n}(t,\cdot)-\widetilde{c}\|_{\infty} = 0
\end{equation}
\noindent with probability at least $1-\delta$. 
\end{theorem}

\begin{proof} 

Let $\mathbb{T}^n$ be the $n$-dimensional torus, and define $\Delta(\gamma) \subset \mathbb{T}^n$ as the set of angle arrays $\left(\theta_1, \ldots, \theta_n\right)$ with the property that there exists an arc of length $\gamma$ containing $\theta_1, \ldots, \theta_n$ in its interior. We start by proving in Lemma \ref{Invariance_property_Sakaguchi_Kuramoto_Oscillators}, given in the Appendix, that with probability $1$ there exists a $\bar{n}_1$ and $0<\gamma <\frac{\pi}{2}-\beta$ such that for all $n>\bar{n}_1$, if $\frac{\pi}{2}>\beta>0$ satisfies \eqref{alpha_condition_for_sycnh}, then $\Delta(\gamma)$ is forward invariant for the dynamics in (\ref{eqn:sampled_dynamics}). 

Note that by Assumption \ref{CSKM_sync_assumption}, for all $x\in[0,1]$, there exists a constant $c$ such that $\lim_{t\to\infty} \|\theta(t,\cdot) -c\|_{L^{\infty}(I)}= 0.$ Thus, there exists $T$ such that for all $t\geq{T}$,
$$
\|{{\theta}}(t,x)-c\|_{L^{\infty}(I)} < \frac{\gamma}{4}
$$
\noindent and in particular 
$$
\|{{\theta}}(T,x)-c\|_{L^{\infty}(I)} < \frac{\gamma}{4}.
$$

\noindent By Theorem \ref{main_theorem}, there exists $\bar{n}_2$ such that for all $n>\bar{n}_2$

$$
\|{{\theta}}^{n}(T,x)-{{\theta}}(T,x)\|_{L^{\infty}(I)} < \frac{\gamma}{4}.
$$
\noindent with probability at least $1-\frac{\delta}{2}$. By the triangle inequality, for all $n>\bar{n}_2$ ,
$$
\|{{\theta}}^{n}(T,x)-c\|_{L^{\infty}(I)} < \frac{\gamma}{2},
$$
with probability at least $1-\frac{\delta}{2}$ for all $x\in[0,1]$. Note that this implies that $\theta^{n}(T)\in\Delta(\gamma)$.

We are considering Erd\H{o}s--R\'enyi random graphs \( G(n,p) \) in the dense regime, and therefore, the edge probability exceeds the connectivity threshold \cite{erdos1960evolution}. As a result, the probability that the Erd\H{o}s--R\'enyi graph is connected converges to $1$ as \( n \to \infty \). Therefore, there exists \( \bar{n}_3 \in \mathbb{N} \) such that for all \( n > \bar{n}_3 \), the probability that the random graph is connected is at least \( 1 - \frac{\delta}{2} \).

Choose $\bar{n} = \max\left({\bar{n}_1},{\bar{n}_2 },{\bar{n}_3}\right)$ and let $n>\bar{n}$. The probability that the random graph is connected and $\theta^{n}(T)\in\Delta(\gamma)$ is at least $1-(\frac{\delta}{2}+\frac{\delta}{2}) = 1-\delta$. 
Since for all $n\geq\bar{n}$, $\theta^{n}(T)\in\Delta(\gamma)$, $\Delta(\gamma)$ is forward invariant, and the random graph is connected with probability at least $1-\delta$, by Theorem 4.1 in \cite{Bullo}, we have that \eqref{freq_convergence_oscillatos_inf_norm} holds with probability at least $1-\delta$.
\end{proof}

\section{Concluding Remarks}
In this work, we compare the solution of a coupled dynamical system over a $W$-random network of size $n$ to the solution of a continuous dynamical system governed by the generating graphon as $n\rightarrow\infty$. Utilizing concentration inequalities for the random adjacency matrix and regularity properties of the graphon, we establish that for finite intervals and large enough sampled graphs, the solutions of the two models remain close in the $L^\infty$ norm, with high probability.  As an application of our main result, we analyze the Kuramoto and Sakaguchi--Kuramoto models on Erd\H{o}s--R\'enyi random graphs. In the homogeneous Kuramoto case, we prove that for $G(n,\alpha_n)$ with $\frac{\log(n)/n}{\alpha_n^3} \to 0$ as $n \to \infty$, the system achieves phase synchronization with high probability, provided the continuum initial condition is continuous and asymptotically synchronizing in the $L^\infty$ norm. For the homogeneous Sakaguchi--Kuramoto model on $G(n,p)$ with fixed $p \in (0,1]$, we establish that frequency synchronization occurs with high probability for sufficiently large $n$, assuming the initial condition satisfies a similar continuity and asymptotic synchronization assumption and the phase shift $\beta$ meets an explicit inequality depending on $p$ and $n$. 


A natural question that arises from this work is whether this analysis can be extended to study synchronization properties of Kuramoto models on random networks beyond Erd\H{o}s-R\'enyi topologies. Another interesting direction is to extend the convergence analysis to sparser graphon regimes, where \( \alpha_n \) decays at the same rate as the Erd\H{o}s--R\'enyi connectivity threshold, \( \log(n)/n \).


\appendix
\numberwithin{equation}{section}
\section{Existence of Solutions for the Continuum Model}\label{appendix:existence-theorem}
We provide the proof of Theorem~\ref{thm:existence}.

\begin{proof}[Proof of Theorem~\ref{thm:existence}]
    \begin{enumerate}
        \item[(i)] This follows from Theorem 3.2 in \cite{Kaliuzhnyi-Medvedev}.
        \item[(ii)] We prove this using a fixed point argument similar to the one used for part (i). Let $I=[0,1]$ and set $X:=C([0,T]\times I;\mathbb{R})$ with the usual norm, $\|u\|_{\infty}:=\max_t\max_x u(t,x)$. For $\eta\in C(I)$, define the integral operator
        \begin{align*}
            [Ku](t,x):=\eta(x)+\int_0^t f(u(s,x))ds+\int_0^t\int_I W(x,y)D(u(s,y)-u(s,x))dyds
        \end{align*}
        First, we note that since $\eta\in C(I)$, $W\in C(I^2,\mathbb{R})$ and $f$ is periodic and continuous (hence bounded), we can conclude $K:X\rightarrow X$. Indeed: 
        \begin{itemize}
            \item For fixed $t$, $[Ku](t,\cdot)$ is continuous in $x$ since
            \begin{enumerate}
                \item[(i)] $\eta$ is continuous in $x$
                \item[(ii)] $\int_0^t f(u(s,x))ds+\int_0^t\int_I W(x,y)D(u(s,y)-u(s,x))dyds$ is continuous in $x$ by Lemma~\ref{lemma:partial-integration-continuity} since the integrand is continuous.
            \end{enumerate}
            \item $[Ku](\cdot,x)$ is Lipschitz continuous in $t$ with Lipschitz constant independent of $x$. This is because by the fundamental theorem of calculus, for fixed $x$, $\partial_t[Ku(t,x)]=f(u(t,x))+\int_I W(x,y)D(u(t,y)-u(t,x))dy$ and the uniform Lipschitz constant is therefore 
            \begin{align*}
                L\leq\max_{x\in I}\max_{t\in[0,T]}|\partial_t[Ku(t,x)]|\leq \|f\|_{\infty}+\|D\|_{\infty}<\infty,
            \end{align*}
            which is independent of $x$.
        \end{itemize}
        By Lemma~\ref{lemma:joint-continuity}, $Ku\in X$.
    \end{enumerate}

    Next, we show that $K$ is Lipschitz and in particular, a contraction for $T<\frac{1}{L_f+2L_D}=:\overline{T}$:
    \begin{align*}
        \|Ku-Kv\|_\infty=&\max_{t\in[0,T]}\max_{x\in I}\Big|\int_0^t f(u(t,x)) dt - \int_0^t f(v(t,x)) dt \\
        &+\int_0^t\int_I W(x,y)D(u(s,y)-u(s,x))dyds- \int_0^t\int_I W(x,y)D(v(s,y)-v(s,x))dyds\Big|\\
        \leq&\max_{t\in[0,T]}\max_{x\in I}\Big[\int_0^t|f(u(t,x))-f(v(t,x))|dt+\int_0^t\int_I|W(x,y)||D(u(s,y)-u(s,x))\\
        &-D(v(s,y)-v(s,x))|dyds\Big]\\
        \leq&\max_{t\in[0,T]}\max_{x\in I}\Big[\int_0^tL_f|u(t,x)-v(t,x)|dt\\
        &+L_D\int_0^t\int_I|u(s,y)-v(s,y)|+|u(s,x)-u(s,x))|dyds\Big]\\
        &\leq(L_f+2L_D)T\|u-v\|_{\infty}.
    \end{align*}
Applying Banach's fixed point theorem, we get the existence of a unique fixed point $\theta\in X$ satisfying $K\theta = \theta$, i.e.,
\begin{align*}
    \theta(t,x)=\eta(x)+\int_0^t f(\theta(s,x))ds+\int_0^t\int_I W(x,y)D(\theta(s,y)-\theta(s,x))dyds.
\end{align*}
Differentiating the above expression with respect to $t$, we see that $\theta$ is the unique solution to~\eqref{eqn:continuum_dynamics} (since we know from part (i) that the solution is unique). Finally, since
\begin{align*}
    \partial_t\theta(t,x) = f(\theta(t,x))+\int_I W(x,y)D(\theta(t,y)-\theta(t,x))dy
\end{align*}
we may conclude that $\partial_t\theta\in C([0,T]\times[0,1],\mathbb{R})$ since $f(\theta(t,x))$ is the composition of two continuous functions and the second term is continuous due to Lemma~\ref{lemma:partial-integration-continuity}.
\end{proof}

\begin{remark}
    Global-in-time existence for~\eqref{eqn:continuum_dynamics} follows from Theorem~\ref{thm:existence} noticing that the maximal time interval of existence, $\overline{T}$, only depends on the Lipschitz constants of $f$ and $D$. Thus the solution can be extended indefinitely by restarting the flow at $t=\overline{T}$, $t=2\overline{T}$ and so on.
\end{remark}

\begin{lemma}\label{lemma:partial-integration-continuity}
    Let $U,V\subset\mathbb{R}^n$ be non-empty, bounded open sets and suppose $g:\overline{U\times V}\rightarrow\mathbb{R}$ is continuous. Then, $h:\overline{U}\rightarrow\mathbb{R}$ defined by $h(x):=\int_V g(x,y) dy$ is continuous.
\end{lemma}
\begin{proof}
    Let $x_0\in \overline{U}$ and let $\epsilon>0$. Since $g$ is continuous, we can find $\delta$ such that $|x-x_0|\leq\delta\implies|g(x,y)-g(x_0,y)|\leq \frac{\epsilon}{|V|}$, then for all $x$ such that $|x-x_0|\leq\delta$ we have
    \begin{align*}
        |h(x)-h(x_0)|&=\left|\int_V g(x,y)dy-\int_V g(x_0,y)dy\right|\\
        &=\left|\int_V g(x,y)-g(x_0,y)dy\right|\\
        &\leq\int_V |g(x,y)-g(x_0,y)|dy\\
        &\leq\frac{\epsilon}{|V|}\int_V dy = \epsilon.
    \end{align*}
\end{proof}

\begin{lemma}\label{lemma:joint-continuity}
    Suppose that $g:[0,T]\times I\rightarrow \mathbb{R}$ satisfies
    \begin{enumerate}
        \item[(i)] $g(t,\cdot)$ is continuous in $x$ for all fixed $t$,
        \item[(ii)] $g(\cdot,x)$ is Lipschitz continuous in $t$ with constant $L$ independent of $x$,
    \end{enumerate}
    then $g$ is jointly continuous.
\end{lemma}
\begin{proof}
    Fix $(t,x)$ and let $\epsilon>0$. Let $\delta_1(t)$ be such that $|f(t,x')-f(t,x)|\leq\epsilon/2$ and let $\delta:=\min\{\delta_1(t),\frac{\epsilon}{2L}\}$.  Then $|(t',x')-(t,x)|\leq \delta\implies |t'-t|\leq\frac{\epsilon}{2L}$ and $|x'-x|\leq\delta(t)$. We then see that
    \begin{align*}
        |(t',x')-(t,x)|\leq \delta\implies |f(t',x')-f(t,x)|&\leq |f(t',x')-f(t,x')|+|f(t,x')-f(t,x)|\\
        &\leq L|t-t'|+\epsilon/2\leq \epsilon.
    \end{align*}
\end{proof}

We state and prove the following lemma that we use in the main text,
\begin{lemma}\label{lemma:difference-of-solutions-continuity}
    Let $\bar\theta^n(t,x)$ denote the solution to~\eqref{averaged_dynamics} expressed as in~\eqref{eq:averaged-solution-piecewise-interpolant} and let $\theta(t,x)$ denote the solution to~\eqref{eqn:continuum_dynamics}. Then the function $f(t):=\max_{x\in I}|\bar\theta^n(t,x)-\theta(t,x)|$ is continuous in $t$.
\end{lemma}
\begin{proof}
    By Theorem~\ref{thm:existence}, $\theta(t,x)$ is continuous in $[0,T]\times I$. Since $\bar\theta^n$ is defined as a step function whose value on each subinterval $I^n_i$ is given by the solution $\theta^n_i(t)$ to the ODE system,~\eqref{averaged_dynamics}, it is continuous on $[0,T]\times I^n_i$ for each $i=1,...,n$. Now define $f_i(t,x)=|\bar\theta^n(t,x)-\theta(t,x)|:[0,T]\times I_i^n\rightarrow\mathbb{R}$ and note that $f_i$ is continuous on $[0,T]\times I^n_i$. By Berge's maximum theorem, $f_i(t):=\max_{x\in I_i} f_i(t,x)$ is continuous on $[0,T]$. Since $f(t)=\max_i f_i(t)$, we conclude that $f(t)$ is also continuous.
\end{proof}
\section{Concentration of Sampled Adjacency Matrix}
\begin{lemma}\label{Upper Bounding Sampled Adjacency Matrix Part 1}
Let \( A^{(n)}_{ij} \sim \operatorname{Ber}\left(\alpha_n W^{(n)}_{ij}\right) \), and define
\[
T_{ij} = \left(\frac{A^{(n)}_{ij}}{\alpha_n} - W^{(n)}_{ij}\right) \beta_{ij}, \quad T_i = \frac{1}{n} \sum_{j=1}^n T_{ij},
\]
where \( \beta_{ij} \in \mathbb{R} \) are deterministic coefficients satisfying \( |\beta_{ij}| \leq M \) for all \( i, j \). Suppose \( \frac{\log(n)/n}{\alpha_n} \to 0 \) as \( n \to \infty \) and \( \alpha_n \leq 1 \). Then for all \( \epsilon, \delta > 0 \), there exists \( \bar{n} \) such that for all \( n > \bar{n} \),
\[
\mathbb{P}\left[\left|T_i\right| \leq \epsilon, \ \forall i\right] \geq 1 - \delta.
\]
\end{lemma}

\begin{proof} 
Note that for fixed \( i \), \( \{T_{ij}\}_{j=1}^n \) are independent, zero-mean random variables, implying \( \mathbb{E}[T_i] = 0 \). Moreover, for all \( j \),
\[
|T_{ij}| = \left| \frac{A^{(n)}_{ij}}{\alpha_n} - W^{(n)}_{ij} \right| |\beta_{ij}| \leq \frac{1}{\alpha_n} |\beta_{ij}| < \frac{M}{\alpha_n}.
\]

Bounding the variance of $T_{ij}$:
\begin{align*}
\mathbb{V}[T_{ij}] \leq \mathbb{E}[T_{ij}^2] &= \mathbb{E}\left[\left( \frac{A^{(n)}_{ij}}{\alpha_n} - W^{(n)}_{ij} \right)^2 \right] \cdot (\beta_{ij})^2 \\
&= \left[ \alpha_n W^{(n)}_{ij} \left( \frac{1}{\alpha_n} - W^{(n)}_{ij} \right)^2 + (1 - \alpha_n W^{(n)}_{ij}) (W^{(n)}_{ij})^2 \right] \cdot (\beta_{ij})^2 \\
&\leq \left[ \alpha_n \left(\frac{1}{\alpha_n} - W^{(n)}_{ij} \right)^2 + \left(1 - \alpha_n W^{(n)}_{ij} \right) (W^{(n)}_{ij})^2 \right] M \\
&\leq \left[ \frac{\alpha_n}{\alpha_n^{2}} \left(1 - \alpha_n W^{(n)}_{ij} \right)^2 + \left(1 - \alpha_n W^{(n)}_{ij} \right) W^{(n)}_{ij} \right] M \\ 
&= \left[\left(1-\alpha_nW^{(n)}_{{ij}}\right) \left(\frac{1}{\alpha_n^{}} \left(1 - \alpha_n W^{(n)}_{ij} \right) +  W^{(n)}_{ij}\right) \right] M\\
&\leq  \left[\frac{1}{\alpha_n} - W^{(n)}_{ij} + W^{(n)}_{ij} \right] M = \frac{M}{\alpha_n}.
\end{align*}

Using Bernstein's inequality:

\[
\mathbb{P}\left[\left|T_i\right|\geq t\right] = \mathbb{P}\left[\left|\frac{1}{n} \sum_{j=1}^{n} T_{i j}\right| \geq t\right] =  \mathbb{P}\left[\left|\sum_{j=1}^{n} T_{i j}\right| \geq tn\right] \leq 2 \exp \left(-\frac{\frac{1}{2}  t^{2} n^2}{n  \frac{M}{\alpha_n}+\frac{1}{3} \frac{M}{\alpha_n}  t n}\right).
\]

Set $t^{2}$ as follows: 

\begin{align*}
t^2=\frac{4 M \log (2 n / \delta)}{n  \alpha_n}  =4 M\left(\frac{\log (n)+\log (2 / \delta)}{n \alpha_n}\right) =4 M\left(\frac{\log (n)}{n \alpha_n}+{\frac{\log (2 / \delta)}{n\alpha_n}}\right).
\end{align*}

As $n\rightarrow{\infty}$, $\frac{\log (n)}{n \alpha_n},\frac{\log (2 / \delta)}{n\alpha_n}\rightarrow{0}$. Hence $t^2\rightarrow{0}$ as $n\rightarrow{\infty}$. Set $n$ large enough such that $t<3$. Then, 

$$
\mathbb{P}\left[\left|T_i\right| \geq t\right] \leq 2 \exp \left(-\frac{\frac{1}{2} t^2 n^2}{2 \frac{n M}{\alpha_n}}\right)=2 \exp \left(-\frac{1}{4 M} \frac{t^2 n^2}{\frac{n}{\alpha n}}\right).
$$

Plugging in $t^2$: 

$$
\mathbb{P}\left[\left|T_i\right| \geq t\right] \leq \frac{\delta}{n}.
$$

By the union bound, for $n$ sufficiently large, with probability at least $1-\delta$, 
$$\left|T_i\right| \leq t=\sqrt{\frac{4 M \log (2 n / \delta)}{n \alpha_n}} \leq \epsilon\quad{\forall{i}}.$$
\end{proof}

\begin{lemma}\label{Upper Bounding Sampled Adjacency Matrix Part 2}
For all $\epsilon,\delta>0$ there exists an $\bar{n}$ such that for all $n>\bar{n}$ with probability $1-\delta:$
$$
\frac{1}{n \alpha_n} \sum_{j=1}^n A^{(n)}_{i j} \leq(1+\epsilon) \quad \text{ for all $i$}.
$$
\end{lemma}

\begin{proof}
This proof follows from Lemma \ref{Upper Bounding Sampled Adjacency Matrix Part 1}. Set $\beta_{ij}=1$ for all $i,j$. With probability at least $1-\delta$, there exists a $\bar{n}$ such that for all $n>\bar{n}$ and all $i$:

\begin{align*}
\left|T_i\right| \leq \epsilon  & \Rightarrow\left|\frac{1}{n} \sum_{j=1}^{n}\left(\frac{A^{(n)}_{i j}}{\alpha_n}-W^{(n)}_{i j}\right)\right| \leq \epsilon \Rightarrow \frac{1}{n \alpha_n} \sum_{j=1}^{n} A^{(n)}_{i j} \leq \epsilon+\frac{1}{n} \sum_{j=1}^{n} W^{(n)}_{i j} \leq 1+\epsilon.
\end{align*}
\end{proof}
\section{Positive Dynamical System}

\begin{lemma}\label{Positive Dynamical System}
Suppose \( \dot{u}(t) \leq_e B u(t) + \omega(t) \), with \( u(0) = u_0 \in \mathbb{R}^n \), and \( B \geq_e \mathbf{0} \). Define 
\[
\dot{s}(t) = B s(t) + \omega(t), \quad s(0) = u_0.
\]
Then for all \( t \geq 0 \), 
\[
u(t) \leq_e s(t).
\]
\end{lemma}

\begin{proof} Let $x(t)=s(t)-u(t)$ and observe,
\begin{equation}\label{positive_dyn_system}
\dot{x}(t) = \dot{s}(t) - \dot{u}(t) \geq_e B s(t) + \omega(t) - B u(t) - \omega(t) = B x(t),
\end{equation}
with \( x(0) = u_0 - u_0 = 0 \).

We show that \eqref{positive_dyn_system} is a positive system by proving that $\mathbb{R}^{n}_{\geq{0}}$ is positively invariant \cite{hirsch2012differential}. Suppose that there exists a $t$ and $i$ such that $x_{i}(t)=0$, then 

$$\dot{x}_i(t) \geq[B x(t)]_i=\sum_{j = 1}^{n} \underbrace{B_{i j}}_{\geq{0}}\underbrace{x_j(t)}_{\geq{0}} \geq 0 .$$

Hence, $x_i(t)$ cannot become negative. Thus, $x(t)=s(t)-u(t)\geq_{e} \mathbf{0}$ implying $u(t)\leq_{e} s(t)$ for all $t$. 
\end{proof}

\section{Element-wise inequalities for matrices}
\begin{lemma}\label{Infinity Norm Bounds for Nonnegative Matrices}
    Suppose $B\geq_{e}\mathbf{0}$. Then for a non-negative integer, $k$
\begin{enumerate}
        \item $\left\|B^k\right\|_{\infty} \leqslant\|B\|_{\infty}^k$, and 
        \item  $\left\|e^B\right\|_{\infty} \leq e^{\|B\|_{\infty}}$.
    \end{enumerate}
\end{lemma}

\begin{proof}
First, note that $B\geq_{e}\mathbf{0}\implies B^{2}\geq_{e} \mathbf{0}$. We proceed by induction on \( k \).

Base case (\( k = 1 \)): This is immediate, since
\[
\|B^1\|_{\infty} = \|B\|_{\infty} = \|B\|_{\infty}^1.
\]

Inductive step: Assume the result holds for some \( k \geq 1 \), i.e.,
\[
\|B^k\|_{\infty} \leq \|B\|_{\infty}^k.
\]

We show it holds for \( k+1 \). Note that
\begin{align*}
\|B^{k+1}\|_{\infty}
&= \|B^k B\|_{\infty}
= \max_{i} \sum_{j=1}^{n} \left[ B^k B \right]_{ij} \\
&= \max_{i} \sum_{j=1}^{n} \sum_{\ell=1}^{n} B^k_{i\ell} B_{\ell j}
= \max_{i} \sum_{\ell=1}^{n} B^k_{i\ell} \left( \sum_{j=1}^{n} B_{\ell j} \right) \\
&\leq \max_{i} \sum_{\ell=1}^{n} B^k_{i\ell} \|B\|_{\infty}
= \|B\|_{\infty} \cdot \max_{i} \sum_{\ell=1}^{n} B^k_{i\ell}
= \|B\|_{\infty} \cdot \|B^k\|_{\infty} \\
&\leq \|B\|_{\infty} \cdot \|B\|_{\infty}^k
= \|B\|_{\infty}^{k+1}.
\end{align*}

By induction, the result holds for all \( k \geq 1 \). Now, using this result,
    \begin{align*}
    \left\|e^B\right\|_{\infty} = \max_{i}\left[\sum_{k=0}^{\infty} \frac{1}{k!} B^k \mathbf{1}\right]_i \leq \sum_{k=0}^{\infty} \frac{1}{k!} \max_{i_k}\left[B^k \mathbf{1}\right]_{i_k} = \sum_{k=0}^{\infty} \frac{1}{k!} \|B^{k}\|_{\infty} \leq \sum_{k=0}^{\infty} \frac{1}{k!} \|B^{}\|_{\infty}^{k}= e^{\|B\|_{\infty}}.
    \end{align*}
 \end{proof}


\begin{lemma}\label{Monotonicity of Matrix Powers and Exponential}

If $\mathbf{0} \leq_e M_1 \leq_e M_2$, then 
\begin{enumerate}
    \item $M_1^k \leq_{e} M_2^k$ for all $k$, and
    \item $e^{M_1} \leq_e e^{M_2}$.
\end{enumerate}
    
\end{lemma}

\begin{proof}
We proceed by induction on \( k \). The base case, \( k = 1 \), holds by assumption: $M_1^1 = M_1 \leq_e M_2 = M_2^1$. For the inductive step, suppose that \( M_1^k \leq_e M_2^k \) for some \( k \geq 1 \). We want to show that \(M_1^{k+1} \leq_e M_2^{k+1} \). Note that $M_1^{k+1} = M_1^k M_1, M_2^{k+1} = M_2^k M_2$.
We use the following fact: If \( A, B, C \geq_e 0 \) and \( A \leq_e B \), then
\[
AC \leq_e BC \quad \text{and} \quad CA \leq_e CB.
\]
This follows from: $(AC)_{ij} = \sum_{\ell = 1}^{n} A_{i\ell} C_{\ell j} \leq \sum_{\ell=1}^{n} B_{i\ell} C_{\ell j} = (BC)_{ij}$ since all terms are nonnegative. Applying this to our case:
\[
M_1^k M_1 \leq_e M_2^k M_1 \leq_e M_2^k M_2.
\]
The first inequality holds by the inductive hypothesis \( M_1^k \leq_e M_2^k \) and \( M_1 \geq_e 0 \), and the second by \( M_1 \leq_e M_2 \) and \( M_2^k \geq_e 0 \). Thus, $M_1^{k+1} \leq_e M_2^{k+1}$. By induction, we conclude that $M_1^k \leq_e M_2^k$ for all $k \in \mathbb{N}.$

Using this fact, 
$$e^{M_1}=\sum_{k=0}^{\infty} \frac{1}{k!}\left(M_1\right)^k \leq_{e} \sum_{k=0}^{\infty} \frac{1}{k!}\left(M_2\right)^k=e^{M_2}.$$
\end{proof}

\section{Upper Bound over Finite Time Interval on Sampled Network}

\begin{lemma}\label{Upper Bound over Finite Time Interval on Sampled Network}
Let $Y_i=\int_0^T\left(g_{n,i}(\tau)\right)^2 d \tau$ where $g_{n,i}(\cdot)$ is defined in Proposition \ref{convgence_sampled_avged}. If $\frac{\log (n) / n}{\alpha_n^3} \rightarrow 0$, then for all $\epsilon,\delta>0$ there exists $\bar{n}$ such that for all $n\geq\bar{n}$, 
$$\mathbb{P}[Y_i \leq \epsilon, \forall{i}] \geq 1-\delta.$$
\end{lemma}

\begin{proof}
    Fix any $h\in\{1,\dots,n\}$ and note, 
    \begin{align*}
    Y_h=\int_0^T\left(g_{n,h}(t)\right)^2 dt &= \int_0^T\left(\frac{1}{n} \sum_{j=1}^{n}\underbrace{\left(\frac{A^{(n)}_{h j}}{\alpha_n}-W^{(n)}_{h j}\right)}_{X^{h}_j} \underbrace{D\left(\bar{\theta}_j(t)-\bar{\theta}_h(t)\right)}_{\beta_{h j}(t)}\right)^2 d t \\
    &= \int_0^T\left(\frac{1}{n} \sum_{j=1}^{n} X^{h}_j \beta_{h j}(t)\right)^2 d t\\
    &= \int_0^T \frac{1}{n^2} \sum_{j=1}^{n} \sum_{k=1}^{n} X^{h}_j X^{h}_k  \beta_{h j}(t) \beta_{h k}(t) d t \\
    &=  \frac{1}{n^2} \sum_{j=1}^{n} \sum_{k=1}^{n} X^{h}_j X^{h}_k  \underbrace{\int_0^T \beta_{h j}(t) \beta_{h k}(t) d t}_{\gamma^{h}_{jk}}\\
    & = \frac{1}{n^2} \sum_{j=1}^{n} \sum_{k=1}^{n} X^{h}_j X^{h}_k  \gamma^{h}_{j k}=: f^{h}\left(X_1^{h}, \ldots, X_n^{h}\right)
    \end{align*}

where $\gamma^{h}_{jk} = \gamma^{h}_{kj}$, $|\gamma^{h}_{jk}|\leq{T}$,$|X^{h}_j|\leq \frac{1}{\alpha_n}$, and $\mathbb{E}\left[(X^h_j)^2\right]\leq\frac{1}{\alpha_n}$ (see Lemma \ref{Upper Bounding Sampled Adjacency Matrix Part 1}). Note, 

$$\mathbb{E}\left[Y_h\right]=\frac{1}{n^2} \sum_{j=1}^{n} \sum_{k=1}^{n} \mathbb{E}\left[X^{h}_j X^{h}_k\right] \gamma^{h}_{j k}=\frac{1}{n^2} \sum_j \mathbb{E}\left[(X^{h}_j)^2\right] \gamma^{h}_{j j}\leq\frac{1}{n^2}  \frac{n}{\alpha_n}  T=\frac{T}{\alpha_n  n}.$$

So, $\mathbb{E}\left[Y_h\right]\rightarrow{0}$ as  $n\rightarrow{\infty}$ since $\frac{1}{n\alpha_n}\leq\frac{\log(n)}{n\alpha^{3}_n}\rightarrow{0}$. We next study 
the concentration of $Y_h$ around it's mean. To this end, we use McDiarmid's Inequality (Bernstein's form), and therefore need to compute:

\begin{itemize}
    \item $B^h:=\max _{i} \sup _{X^h_i}\left|f\left(X^h_1, \ldots,X^h_i, \ldots, X^h_n\right)-\mathbb{E}_{X_i} f\left(X^h_1, \ldots, X^h_i,\ldots, X^h_n\right)\right|$, 
    \item $V^h_i:=\sup _{X^h_i} \mathbb{E}_{X^h_i}\left(f\left(X^h_1, \ldots, X^h_i, \ldots, X^h_n\right)-\mathbb{E}_{X^h_i} f\left(X^h_1, \ldots, X^h_i, \ldots, X^h_n\right)\right)^2$, and
    \item ${(\sigma^h)}^2:=\sum_{i=1}^n V_i^{h}$.
\end{itemize}

In the following we omit the superscript $h$ for simplicity. To compute $B, V_i$, and $\sigma^2$, note that for all $i$,
\begin{align*}
    &f^{}\left(X_1, \ldots, X_n\right)  = \frac{1}{n^2}\left[\sum_{k=1}^{n} X_i X_k \gamma^{}_{ik}+\sum_{j \neq i} \sum_{k=1}^{n} X_j X_k \gamma^{}_{j k}\right] =\frac{1}{n^2}\left[X_i^2 \gamma_{i i}+2 \sum_{k \neq i} X_i X_k \gamma_{i k}+\sum_{j \neq i} \sum_{k \neq i} X_j X_k \gamma^{}_{j k}\right]
\end{align*}

Hence, 
$$\mathbb{E}_{X_i}[f(X)]=\frac{1}{n^2}\left[\mathbb{E}\left[X_i^2\right] \gamma_{i i}+\sum_{j \neq i} \sum_{k \neq i} X_j X_k \gamma_{j k}\right]$$ 

implies,

$$f(X)-\mathbb{E}_{X_i}[f(X)] = \frac{1}{n^2}\left[X_i^2 \gamma_{i i}+2 \sum_{k \neq i} X_i X_k \gamma_{i k}-\mathbb{E}\left[X_i^2\right] \gamma_{i i}\right].$$

Thus, 

$$
B \leq \frac{1}{n^2}\left[\frac{T}{\alpha_n^2}+2(n-1) \frac{T}{\alpha_n^2}+\frac{T}{\alpha_n^2}\right]=\frac{2 n  T}{n^2 \alpha_n^2}=\frac{2 T}{n \alpha_n^2}.
$$

Moreover,
\begin{align*}
\mathbb{E}_{X_i}\left(f(X)-\mathbb{E}_{X_i} f(X)\right)^2 &= \frac{1}{n^4}\mathbb{E}_{X_i}\left[\left(X_i^2 \gamma_{i i}+2 \sum_{k \neq i} X_i X_k \gamma_{i k}-\mathbb{E}\left[X_i^2\right] \gamma_{i i}\right)^2\right]\\
&= \frac{1}{n^4} \mathbb{E}_{X_i} \bigg[
    X_i^4 \gamma_{ii}^2
  + \left(2 \sum_{k \neq i} X_i X_k \gamma_{ik}\right)^2
  + \left(\mathbb{E}[X_i^2] \gamma_{ii}\right)^2 \\
&\qquad
  + 4 X_i^2 \gamma_{ii} \left(\sum_{k \neq i} X_i X_k \gamma_{ik}\right)
  - 2 X_i^2 \gamma_{ii}^2 \mathbb{E}[X_i^2]
  - 4 \mathbb{E}[X_i^2] \gamma_{ii} \sum_{k \neq i} X_i X_k \gamma_{ik}
\bigg]\\
&= \frac{1}{n^4} \left[\gamma_{i i}^2 \mathbb{E}\left[X_i^4\right]+4 \mathbb{E}\left[X_i^2\right]\left(\sum_{k \neq i} X_k \gamma_{i k}\right)^2-\mathbb{E}\left[X_i^2\right]^2 \gamma_{i i}^2 +4 \mathbb{E}\left[X_i^3\right] \gamma_{i i}\left(\sum_{k \neq i} X_k \gamma_{i k}\right)-0\right]\\
&\leq \frac{T^2}{n^4}\left[\frac{2}{\alpha_{n}^3}+ \frac{4\cdot2}{\alpha_{n}}\left(n \frac{1}{\alpha_{n}}\right)^2+\left(\frac{2}{\alpha_{n}}\right)^2+ \frac{4 \cdot 2}{\alpha_{n}^2}\left(\frac{1}{\alpha_{n}}\right)n\right]\\
&\leq 8 \frac{T^2}{n^4}\left[\frac{1}{\alpha_{n}^3}+\frac{1}{\alpha_{n}}\left(n \frac{1}{\alpha_{n}}\right)^2+\left(\frac{1}{\alpha_{n}}\right)^2+\frac{1}{\alpha_{n}^2}\left(\frac{1}{\alpha_{n}}\right)n\right]\\
&= \frac{8 T^2}{n^4}\left[\frac{1}{\alpha_{n}^3}+\frac{1}{\alpha_{n}^3} \cdot n^2+\frac{1}{\alpha_{n}^2}+\frac{1}{\alpha_{n}^3} n\right] \leq \frac{8\cdot{4}T^{2}}{n^2\alpha_{n}^{3}}.
\end{align*}
 
where, we used the following result: 

\begin{align*}
\left|\mathbb{E}\left[X_i^l\right]\right| & =\left|\mathbb{E}\left[\left(\frac{A^{(n)}_{hi}}{\alpha_n}-W^{(n)}_{hi}\right)^l\right]\right| \\
&=\left|{\alpha_n} {W}^{(n)}_{hi}\left(\frac{1}{\alpha_n}-W^{(n)}_{hi}\right)^l+\left(1-\alpha_n W^{(n)}_{hi}\right)\left(-W^{(n)}_{hi}\right)^l\right|\\
&\leq\frac{\alpha_n}{\alpha_n^l} W^{(n)}_{hi}(1-\alpha_n W^{(n)}_{hi})^l+{\left(1-\alpha_n W^{(n)}_{hi}\right)(W_{hi}^{(n)})^{l}}\\
&\leq \frac{1}{\alpha_n^{l-1}}+1=\frac{1+\alpha_n^{l-1}}{\alpha_n^{l-1}} \leq \frac{2}{\alpha_n^{l-1}} .
\end{align*}

Overall, we obtain : 
$$
(\sigma^h)^2=\sum_{i=1}^{n} V^h_i \leq \frac{K}{n \alpha_{n}^3}, \text{ and } B^h \leq \frac{2 T}{n \alpha_{n}^2},
$$

where $K = 32T^{2}$. Now, using McDiarmid's Inequality (Bernstein's form), for any fixed $h\in\{1,\cdots,n\}$ and any  $t<\frac{3}{\alpha_n}$: 

\begin{align*}
\mathrm{P}\left(f^h\left(X^h_1, \ldots, X^h_n\right)-\mathbb{E}\left[f^h\left(X^h_1, \ldots, X^h_n\right)\right] \geq t\right) \leq \exp \left(-\frac{t^2}{2\left(\frac{K}{\alpha_{n}^3 n}+\frac{2T}{n \alpha_{n}^2} \cdot \frac{t}{3}\right)}\right) \leq \exp \left(-\frac{t^2}{(2K+4T) \frac{1}{\alpha_{n}^3 n}}\right).
\end{align*}

Setting $t^{2} =\frac{\log (n / \delta)(2K+4T)}{\alpha_n^3 n}$, 

\begin{align*}
\mathrm{P}\left(f^h\left(X^h_1, \ldots, X^h_n\right)-\mathbb{E}\left[f^h\left(X^h_1, \ldots, X^h_n\right)\right] \geq t\right)\leq \frac{\delta}{n}.
\end{align*}

Recall $Y^h = f^h\left(X^h_1, \ldots, X^h_n\right)$. By the union bound, over $h\in\{1,\dots{n}\}$

\begin{align*}
\mathrm{P}\left(Y^h-\mathbb{E}\left[Y^h\right] < t \quad \forall{h}\right)\geq 1-\delta.
\end{align*}

With this probability, $Y^h\leq\mathbb{E}\left[Y^h\right]+t$. We obtain the desired result since when $\frac{\log(n)/n}{\alpha_n^{3}}\rightarrow{0}$ as $n\rightarrow{\infty}$ the $\mathbb{E}\left[Y^h\right]+t$ can be made arbitrarily small (less than $\epsilon$) for $n$ sufficiently large. 

\end{proof}

\section{Invariance Property for Sakaguchi-Kuramoto Oscillators on Erd\H{o}s-R\'enyi Random Graphs}
\begin{lemma}\label{Invariance_property_Sakaguchi_Kuramoto_Oscillators}
Fix $p\in(0,1]$, $f=0$, $W=p$, and $D(\cdot) = \sin(\cdot+\beta)$. With probability $1$ there exists a $\bar{n}$ such that for all $n>\bar{n}$, if $\frac{\pi}{2}>\beta>0$ satisfies \eqref{alpha_condition_for_sycnh}, then there exists a $0<\gamma <\frac{\pi}{2}-\beta$ such that: If $|{\theta}_i^{n}(0)-{\theta}_j^{n}(0)|<\gamma$ for all $i,j$, then $|{\theta}_i^{n}(t)-{\theta}_j^{n}(t)|<\gamma$ for all $t\geq{0}$ for the dynamics in (\ref{eqn:sampled_dynamics}).
\end{lemma}
 
\begin{proof}
For fixed $W=p\in(0,1]$, $A^{(n)}$ is such that $A_{ij}^{(n)} = A_{ji}^{(n)}= \textrm{Ber}(p)$. By the Chernoff-Hoeffding inequality \cite{Medvedev_main}, for any $i\in\{1,\dots,n\}$

\begin{align*}
    P\left(\left|\sum_{k=1}^{n}A^{(n)}_{ik}-np\right|\geq n^{\frac{2}{3}}p\right)\leq 2\exp\left(\frac{{-2\left(n^{2/3}p\right)^{2}}}{n}\right) = 2\exp\left(-2p^{2}n^{1/3}\right). 
\end{align*}
\noindent By the union bound, 
\begin{align*}
    P\left(\left|\sum_{k=1}^{n}A^{(n)}_{ik}-np\right|< n^{\frac{2}{3}}p, \quad\forall{i} \right)\geq 1-2n\exp^{-2p^{2}n^{1/3}}. 
\end{align*}

\noindent Since $\sum_{n}2n\exp\left(-2p^{2}n^{1/3}\right)<\infty$, by the Borel–Cantelli lemma, with probability $1$, there exists a $\bar{n}_1$ such that for all $n\geq\bar{n}_1$,

\begin{align}\label{first_bound_invariance_property}
\left|\sum_{k=1}^{n}A^{(n)}_{ik}-np\right|\leq n^{\frac{2}{3}}p \text{ $\forall{i}$ } \implies \sum_{k=1}^{n}A^{(n)}_{ik}\leq p (n+n^{2/3}) \text{ $\forall{i}$ }.
\end{align}

\noindent Next note that for fixed $i\neq j$, $M_k= \min(A^{(n)}_{ik},A^{(n)}_{jk})$ is a Bernoulli random variable with mean $p^2$. Moreover, $M_k$ are independent for each $k\neq{i,j}$. Hence, by the Chernoff-Hoeffding inequality \cite{Medvedev_main} once again, 

\begin{align*}
    P\left(\left|\sum_{k\neq{i,j}} \min\left(A^{(n)}_{ik},A^{(n)}_{jk}\right)-(n-2)p^{2}\right|\geq n^{\frac{2}{3}}p^{2}\right)\leq 2\exp^\frac{{-2\left(n^{2/3}p^{2}\right)^{2}}}{n-2} \leq  2\exp^{-2p^{4}n^{1/3}}. 
\end{align*}

\noindent By the union bound

\begin{align*}
    P\left(\left|\sum_{k\neq{i,j}} \min\left(A^{(n)}_{ik},A^{(n)}_{jk}\right)-(n-2)p^{2}\right|< n^{\frac{2}{3}}p^{2},\quad\forall{i\neq{j}}\right)\geq 1- 2n^{2}\exp^{-2p^{4}n^{1/3}}. 
\end{align*}

\noindent By the Borel–Cantelli lemma, with probability $1$, there exists a $\bar{n}_2$ such that for all $n\geq\bar{n}_2$, and for all $i\neq{j}$,

\begin{align}\label{second_bound_invariance_property}
  \left|\sum_{k\neq i,j}\min\left(A^{(n)}_{ik},A^{(n)}_{jk}\right)-(n-2)p^{2}\right|\leq n^{\frac{2}{3}}p^{2} 
  \implies  \sum_{k=1}^{n} \min\left(A^{(n)}_{ik},A^{(n)}_{jk}\right)\geq \sum_{k\neq i,j}\min\left(A^{(n)}_{ik},A^{(n)}_{jk}\right) \geq p^2(n-n^{2/3}-2).
\end{align}

\noindent Putting \eqref{first_bound_invariance_property} and \eqref{second_bound_invariance_property} together, with probability $1$ there exists a $\bar{n} = \max\{\bar{n}_1,\bar{n}_2\}$ such that 

\begin{align*}
  \frac{\sum_{k=1}^{n}A^{(n)}_{ik}+\sum_{k=1}^{n}A^{(n)}_{jk}}{\sum_{k=1}^{n} \min\left(A^{(n)}_{ik},A^{(n)}_{jk}\right)}\leq\frac{2(n+n^{2/3})}{p(n-2-n^{2/3})}  \text{ $\forall{i\neq{j}}$.}
\end{align*}

\noindent Then, by assumption, 
\begin{equation}\label{assumed_alpha_condition}
\frac{\cos^{2} (\beta)}{\sin^{}(\beta)} > \frac{2}{p} \left(\frac{1+\frac{1}{n^{1/3}}}{1-\frac{1}{n^{1/3}}-\frac{2}{n}}\right) = \frac{2(n+n^{2/3})}{p(n-2-n^{2/3})}\geq\frac{\sum_{k=1}^{n}A^{(n)}_{ik}+\sum_{k=1}^{n}A^{(n)}_{jk}}{\sum_{k=1}^{n} \min\left(A^{(n)}_{ik},A^{(n)}_{jk}\right)},\text{ $\forall{i\neq{j}}$.}
\end{equation}

\noindent Since $\cos(\beta)\sin(\frac{\pi}{2}-\beta)=\cos^{2}(\beta)$ and $\sin(\beta)>0$, 

\begin{align}\label{sufficient_condition_invariance_Bullo1}
\frac{\cos^{}(\beta)\sin^{}(\frac{\pi}{2}-\beta)}{\sin^{}(\beta)} > \frac{\sum_{k=1}^{n}A^{(n)}_{ik}+\sum_{k=1}^{n}A^{(n)}_{jk}}{\sum_{k=1}^{n} \min\left(A^{(n)}_{ik},A^{(n)}_{jk}\right)}, \quad \forall{i\neq{j}},
\end{align}

\noindent which yields  
\begin{align}\label{sufficient_condition_invariance_Bullo2}
 -\cos^{}(\beta)\sin\left(\frac{\pi}{2}-\beta\right)\sum_{k=1}^{n} \min\left(A^{(n)}_{ik},A^{(n)}_{jk}\right) + \sin^{}(\beta) \left(\sum_{k=1}^{n}A^{(n)}_{ik}+\sum_{k=1}^{n}A^{(n)}_{jk}\right)<{0}\quad\forall{i\neq{j}}. 
\end{align}
It follows from the proof of Theorem 4.3 of \cite{Bullo} that there exists $\gamma <\frac{\pi}{2}-\beta$ such that, if $|{\theta}_i^{n}(0)-{\theta}_j^{n}(0)|<\gamma \quad{\forall{i,j}}$, \eqref{sufficient_condition_invariance_Bullo2} implies that $|{\theta}_i^{n}(t)-{\theta}_j^{n}(t)|<\gamma \quad \forall{}i,j$ for all $t\geq{0}$, as desired. 
\end{proof}

\section*{Acknowledgments}
We would like to thank a few members of Cornell University's mathematics community for helpful conversations at varying stages in this work: Shawn Ong, Alex Townsend, and Martin Kassabov. S. V. Nagpal thanks the NSF Research Training Group Grant: Dynamics, Probability, and PDEs in Pure and Applied Mathematics, DMS-1645643 for partially funding this work. G. G. Nair thanks the Hausdorff Research Institute for Mathematics in Bonn for its hospitality during the Trimester Program, Mathematics for Complex Materials funded by the Deutsche Forschungsgemeinschaft (DFG, German Research Foundation) under Germany’s Excellence Strategy– EXC-2047/1– 390685813. F. Parise acknowledges that this material is based upon work supported by the National Science Foundation CAREER Grant No. 2340289.

\bibliographystyle{siamplain.bst}
\bibliography{references}
\end{document}